\documentclass[reqno, a4paper, 11pt]{article}
\usepackage[T1]{fontenc}
\usepackage[utf8]{inputenc}
\usepackage[english]{babel}
\usepackage{amsmath}
\usepackage{amssymb}
\usepackage{amsfonts}
\usepackage{amsthm}
\usepackage{bm}
\usepackage{enumerate}
\usepackage[shortlabels]{enumitem}
\usepackage{moreenum}
\usepackage{mathtools}
\usepackage{xcolor}
\usepackage{csquotes}
\usepackage{hyperref}
\usepackage{subfigure}

\numberwithin{equation}{section}

\newtheorem{theorem}{Theorem}[section]
\newtheorem{definition}[theorem]{Definition}
\newtheorem{algorithm}[theorem]{Algorithm}
\newtheorem{lemma}[theorem]{Lemma}
\newtheorem{corollary}[theorem]{Corollary}
\newtheorem{proposition}[theorem]{Proposition}

\newtheorem{remark}[theorem]{Remark}

\newtheorem{assumption}[theorem]{Assumption}

\newcommand{\ii}{i}
\newcommand{\im}{\operatorname{Im}}
\newcommand{\re}{\operatorname{Re}}

\newcommand{\Res}{\operatorname{Res}}
\newcommand{\ma}{\begin{pmatrix}}
	\newcommand{\am}{\end{pmatrix}}
\newcommand{\qqq}{\qquad \qquad}
\newcommand{\qq}{\quad \quad}
\def\supp{\mathop{\rm supp}\nolimits}
\def\Im{{\rm Im\,}}

\begin{document}

	\title{Inverse spectral Love problem via Weyl-Titchmarsh function}
	\author{Samuele Sottile\footnotemark[1]\;\,\footnotemark[2]}
	\renewcommand{\thefootnote}{\fnsymbol{footnote}}
		\footnotetext[1]{Centre for Mathematical Sciences, Lund University, 221 00 Lund, Sweden; \tt\href{mailto:samuele.sottile@math.lu.se}{samuele.sottile@math.lu.se}}
			\footnotetext[2]{Department of Materials Science and Applied Mathematics, Malm\"{o} University, SE-205 06 Malm\"{o}, Sweden; \tt\href{mailto:samuele.sottile@mau.se}{samuele.sottile@mau.se}}
	
	\maketitle
	
	\begin{abstract}
		In this paper we prove an inverse resonance theorem for the half-solid with vanishing stresses on the surface via Weyl-Titchmarsh function. Using a semi-classical approach it is possible to simplify this three-dimensional problem of the elastic wave equation for the half-solid as a Schr{\"o}dinger equation with Robin boundary conditions on the half-line. The goal of the paper is to establish a method to recover the potential from the Weyl-Titchmarsh function for non self-adjoint problems and to establish a one-to-one and onto map between suitable function spaces.
Moreover, we produce an algorithm in order to retrieve the shear modulus from the eigenvalues and resonances, via the spectral data.\\[4mm]
	\end{abstract}

%
%
%
%
%
%
%
%
%
%




\vskip 0.25cm

\section{Introduction.}\label{intro}

\setcounter{equation}{0}
\subsection{Inverse resonance problems, previous results.}

The Love boundary value problem for the vertically inhomogeneous elastic isotropic medium in the half-space (see \cite{4authorsPaper, Sottile})
\begin{equation}\label{love}
-\frac{\partial}{\partial Z} \hat{\mu} \frac{\partial \varphi_2}{\partial Z} + \hat{\mu} \, |\xi|^2 \varphi_2
= \omega^2 \varphi_2 ,
\end{equation}
\begin{equation}\label{love_b}
\frac{\partial \varphi_2}{\partial Z}(0) = 0,
\end{equation}
where $Z \in \left( - \infty, 0 \right]$ is the coordinate with direction normal to the boundary, $\hat{\mu}$ is the density-normalized shear modulus, $\xi$ is the dual of the coordinate vector $(x,y)$ parallel to the boundary, $\omega$ is the frequency and $\varphi_2$ is the component of the displacement vector on the $y$ direction. Equation \eqref{love} describes the motion of an infinitesimal element of elastic solid on a direction lying on a plane parallel to the Earth's surface. The boundary condition \eqref{love_b} says that the infinitesimal element has zero normal velocity on the Earth's surface, which is line with the fact that Love waves are transverse waves. Equations \eqref{love}-\eqref{love_b} are obtained in \cite{4authorsPaper} after decoupling the elastic wave equation for infinitesimal solid and using the semiclassical limit. A change of variable (see Section \ref{Semiclass descr}) in equations \eqref{love}--\eqref{love_b} leads to a Schr{\"o}dinger equation 
\begin{equation}
-u''+Vu= k^2 u, 
\end{equation}
with Robin boundary condition 
\begin{equation}
u'(0) + h u(0)= 0.
\end{equation}
Equations \eqref{love}--\eqref{love_b} describe the equation of motion of Love seismic surface waves. Surface waves are waves who travel close to the Earth's surface with amplitude decaying exponentially with the distance to the Earth's surface. Love waves are a type of seismic surface waves (the other type is the Rayleigh waves) which are transverse, in the sense that the oscillation of an infinitesimal element of elastic solid is perpendicular to the direction of propagation of the wave. The boundary condition \eqref{love_b} is in line with the Love waves being transverse waves, hence not implying oscillations on the vertical axis.

A first complete result on the inverse resonance problem for Love seismic surface waves was obtained in \cite{Sottile}. In \cite{Sottile} the author reconstruct uniquely the potential of a Schr{\"o}dinger equation with Robin boundary condition from eigenvalues and resonances, partly using previous results of Marchenko \cite{Marchenko}, employing a different class of potential and different boundary condition, and Korotyaev \cite{Korotyaev}, employing the same class of potential but different boundary condition. In this paper, we propose an alternative way of reconstructing the potential (and shear modulus) via Weyl-Titchmarsh function formalism, that can be generalized to non self-adjoint problems. This approach is inspired by the paper \cite{Bealsetal1995}, which uses a similar approach for the Rayleigh problem, consisting on a non self-adjoint  Schr{\"o}dinger (matrix-valued) equation with Robin boundary condition. The same is reproduced later by \cite{Iantchenko} for the same Rayleigh system as in \cite{Bealsetal1995}. Some other remarkable paper about surface (Rayleigh) waves are the paper from Pekeris (see \cite{Pekeris}), who obtained uniqueness of density and Lame parameters through the knowledge of the displacements at the boundary and under the hypothesis of them being real analytic; or a paper from Markushevich (see \cite{Markushevich1992}) that obtained uniqueness of the reconstruction of the smooth potential through known boundary values of solutions of the problem for two given independent sources. In \cite{Iantchenko_Rayleigh} they find direct result on location of the resonances for the Rayleigh system.

In relation to seismology, solving the inverse Love problem means reconstructing the parameters that determine the elasticity of the medium in the interior of the Earth from measurements performed on the boundary of the Earth's surface, which are, for example, the frequencies or the wave numbers of surface waves (eigenvalues and resonances). The Earth is a compact domain, but, for simplification, we consider it as a flat half space $\mathbb{R}^2 \times \left(-\infty, 0\right]$.

 There are only a few examples of complete characterizations of inverse resonance problems, for instance by Korotyaev (see \cite{Korotyaev}), who solved it on the half-line for compactly supported potentials with Dirichlet boundary condition, or Christiansen, who solved it on the whole line for step-like potentials (see \cite{Christiansen2005}), using some results from an earlier paper of Cohen--Kappeler (\cite{CohenKappeler1985}). Some other examples are \cite{Borthwick, KorotyaevRotation, KorotyaevDirac}. In \cite{Gesztesy} is shown a characterization result for real integrable potentials in the Schr{\"o}dinger operator on the half line using the Krein spectral shift function, which is connected to the scattering phase $\delta(k)$ through the Birman-Krein formula. In \cite{IantchenkoJacobi} asymptotical values of resonances are obtained for the periodic Jacobi operator with finitely supported perturbations. There are some other examples of inverse problems in Seismology. For example, in \cite{Iantchenko} they analyze an inverse spectral problem for the semiclassical Rayleigh operator starting with spectral data being the Weyl-Titchmarsh matrix.


\subsection{Semiclassical description of the Love waves system.}\label{Semiclass descr}

Starting from equations \eqref{love}--\eqref{love_b} we perform a transformation so that the resulting boundary value problem assumes a Schr{\"o}dinger-type form. Once we have performed the calibration transform, we need to solve an inverse resonance  Schr{\"o}dinger problem with Robin boundary condition, where the eigenvalues and resonances are the poles of the resolvent with respect to the parameter $k$, respectively for $\im k>0$ and $\im k \leq0$.

The main goal of this paper is to retrieve the shear modulus $\hat{\mu}$ ($\hat{\mu}=\mu/\rho$, with $\rho$ being the density)  as we explain in Remark \ref{Remark shear modulus}. This is obtained by an application of a characterization (see Theorem \ref{Theorem Inverse spectral Love problem}) between a class $\mathbb{M}_{x_I}$ of Weyl functions (see Definition \ref{Class of Weyl function}) and a class $\mathbb{V}^1_{x_I}$ of potentials (see Definition \ref{Class of potentials V' in L1}).

We make a simplifying assumption in the following.
\begin{assumption}[Homogeneity]\label{Homogeneity assumption}
	We assume that below a certain depth $Z_I$, the medium is homogeneous, so the shear modulus is constant
	\begin{equation}\label{homogeneity in depth}
	\hat{\mu} \left(Z\right) = \hat{\mu} \left(Z_I \right):=\hat{\mu}_I  \qquad \text{for  } Z \leq Z_I.
	\end{equation}
\end{assumption}

We perform the calibration substitution in \eqref{love}--\eqref{love_b} as
\begin{equation*}
\varphi_2=\frac{1}{\sqrt{\hat{\mu}}} u,\quad  \frac{d}{dZ} \left( \hat{\mu} \frac{d}{dZ} \varphi_2 \right)= \frac{1}{4} \hat{\mu}^{-\frac{3}{2}} (\hat{\mu}')^2u-\frac{1}{2} \hat{\mu}^{-\frac{1}{2}}\hat{\mu}'' u+\hat{\mu}^\frac{1}{2}u''
\end{equation*}
and we get
\begin{equation*}
u''-|\xi|^2u=\left[\frac12\frac{\hat{\mu}''}{\hat{\mu}}-\frac14  \left( \frac{\hat{\mu}'}{\hat{\mu}} \right)^2 -\frac{1}{\hat{\mu}}\omega^2\right]u. 
\end{equation*}
We set the quasi momentum $k:=\sqrt{\frac{\omega^2}{\hat{\mu}_I}-|\xi|^2}$ and  
\begin{equation}\label{Definition of the potential}
V=\frac{1}{2}\frac{\hat{\mu}''}{\hat{\mu}}-\frac{1}{4}  \left( \frac{\hat{\mu}'}{\hat{\mu}} \right)^2 -\frac{1}{\hat{\mu}}\omega^2+\frac{1}{\hat{\mu}_I}\omega^2= \frac{(\sqrt{\hat{\mu}})''}{ \sqrt{\hat{\mu}}}-\frac{1}{\hat{\mu}}\omega^2+\frac{1}{\hat{\mu}_I}\omega^2,
\end{equation}
where $\hat{\mu}_I:=\hat{\mu}(Z_I)$ is the value of the shear modulus at the depth $Z_I$, below which the medium is homogeneous. By Assumption \ref{Homogeneity assumption}, $\hat{\mu} (Z) = \hat{\mu}_I $ constant for $Z \leq Z_I$, hence also the derivatives $\hat{\mu}'$ and $\hat{\mu}''$ vanish for $Z \leq Z_I$. This implies that the potential $V$ has compact support and depends only on $Z$ as we fixed $\omega$ and let our spectral parameter $\xi$ vary. In this way the potential $V= V_{\omega}$ can be parametrized by $\omega$ and we let $\xi$ vary.
\begin{remark}
We will assume that the potential $V \in \mathbb{V}^1_{x_I}$ (Definition \ref{Class of potentials V' in L1}), that implies the shear modulus $\hat{\mu}$ to be constant below the depth $Z_I$ and to be different than $\hat{\mu}_I$ in an interval of type $\left(Z_I, a+Z_I\right)$ for $a>0$.
\end{remark}

The Love scalar equation takes the following form:
\begin{equation}\label{Schrod equation}
-u''+Vu= \lambda u, \qquad \lambda=k^2,
\end{equation}
with corresponding boundary condition that becomes of Robin type after the transformation
\begin{equation}\label{robin bc}
u'(0) + h u(0)= 0, \qquad h= - \frac{1}{2} \frac{\hat{\mu}'(0)}{\hat{\mu}(0)}.
\end{equation}
To resemble the classical formulation, we make the substitution $Z=-x$, which leads the domain to become $\left[ 0 , +\infty \right)$ and we study the problem in terms of $k$. In our case, the potential of the Schr{\"o}dinger operator is real because we are considering an elastic medium. In the case of an inelastic medium, we would have a complex potential that  implies the loss of part of the energy which is converted into heat. We make a self-adjoint realization in $L^2(\mathbb{R}_+)$ of the operator in \eqref{Schrod equation} due to the boundary condition (see \cite{Chadan}). Then the operator appearing on the left hand side of \eqref{Schrod equation}  prescribed with the domain
\begin{align}\label{domain of self adjoint operator}
D=\{ u \in H^2\left[0, + \infty\right) : Vu \in L^2\left[0,+\infty\right), u'(0) + h u(0) = 0\}
\end{align}
and the $L^2$ inner product is self-adjoint.

\subsection{Preliminary results}
In this subsection we list some preliminary results which are needed throughout the whole paper (see \cite{Sottile}).
First, we define  the class $\mathbb{V}^1_{x_I}$, to which the potential belongs throughout the whole paper.
\begin{definition}\label{Class of potentials V' in L1}
	We denote by $\mathbb{V}^1_{x_I}$ the class of real potentials $V$ such that $V,V' \in  L^1(\mathbb{R}_+)$, $\supp V \subset \left[0,x_I\right]$ for some $x_I >0$ and for each $\epsilon >0$ the set $\left(x_I-\epsilon, x_I\right) \cap \supp V$ has positive Lebesgue measure. 
\end{definition}
Below, we define the solution to \eqref{Schrod equation} satisfying the radiation condition.
	\begin{definition}[Jost solution]\label{Jost solution definition}
	The \textbf{Jost solutions} $f^\pm$ are the unique solutions to the differential equation \eqref{Schrod equation} that satisfy the following condition
	\begin{equation}\label{Jost solution condition}
	f^\pm(x,k)=e^{\pm ik x}\qquad \mbox{for} \;\; x>x_I.
	\end{equation}
\end{definition}
The Jost solution satisfies the Volterra-type equation
\begin{equation}\label{Volterra-type eq for f}
f(x,t) = e^{\ii k x} - \int_{x}^{\infty} \frac{\sin \left[ k(x-t) \right]}{k} V(t) f(t,k) dt.
\end{equation}
It is a known that the spectrum for the operator \eqref{Schrod equation} with domain \eqref{domain of self adjoint operator} consists of a finite number of purely imaginary and simple eigenvalues in $k$, there are no real eigenvalues and the eigenfunctions are real (see \cite{ levitaninverse, Sottile}).

From the values of the Jost solution at the boundary we can define the Jost function as follows.

\begin{definition}[Jost function]\label{Jost function definition}
	We define the \textbf{Jost function} $f_h(k)$ of the Schr{\"o}dinger operator $-\frac{d^2}{dx^2} + V$ in \eqref{Schrod equation} with Robin boundary condition \eqref{robin bc} as the quantity
	\begin{equation}\label{Jost function}
		f_h(k) = f(0,k) h + f'(0,k)
	\end{equation}
	where $f(0,k)$ is the Jost solution evaluated at $x=0$.
\end{definition}
Below we define the regular solution (see \cite[Section 1.2]{Chadan}).
\begin{definition}[Regular solution]\label{regular solution varphi}
	We define the \textit{regular solution} $\varphi$ of the Cauchy problem \eqref{Schrod equation} with Robin boundary condition \eqref{robin bc} as the quantity
	\begin{align}\label{equation for def of regular solution}
	\varphi(x,k) = - \frac{1}{2ik} \left[\overline{f_h(k)}f(x,k) - f_h(k) \overline{f(x,k)}\right].
	\end{align}
\end{definition}

We define the eigenvalues as the zeros of the Jost function $f_h$ in the physical sheet $\im k>0$ and we denote as (scattering) resonances the zeros of $f_h$ in the unphysical sheet $\im k<0$. Eigenvalues lead to a $L^2$ solution of the differential equation, while resonances lead to a non $L^2$ solution. We jointly enumerate the zeros of $f_h$, which are eigenvalues and resonances, as $\left(k_j\right)_{j \in \mathbb{N}}$, where $k_1, \cdots,k_N$ are the eigenvalues.

We recall a result of analyticity of Jost solution and Jost function for potential in the class $\mathbb{V}^1_{x_I}$ and some estimates of those. For more details and proofs see \cite{Sottile}
\begin{theorem}\label{Analyticity and continuity of Jost function and solution}
	For each fixed $x\geq0$, the Jost solution $f(x,k)$ and the Jost function $f_h(k)$ are entire in $k$.
\end{theorem}

\begin{lemma}[Uniform bounds on Jost function]\label{uniform bounds on Jost function}
	Let $V \in \mathbb{V}^1_{x_I}$. Then the Jost function is  of exponential type and satisfies the following estimates:
	\begin{align}
	|f(0,k) - 1| &\leq e^{(|\im k| - \im k)x_I}  a e^{a} \label{estimates on Jost solution} \\
	\left|f(0,k)  - 1 + \frac{\hat{V}(0) - \hat{V}(k) }{2ik} \right| &\leq \frac{a^2}{2} e^{(|\im k| - \im k)x_I} e^{a} \label{estimates on Jost solution 2} \\
	\left|f_h(k)- ik \right| &\leq  \left\| V \right\|  e^{(|\im k| - \im k)x_I}  e^{a} \label{bound first iteration Jost function} \\
	\left|f_h(k) -ik - h + \frac{\hat{V}(0) + \hat{V}(k) }{2}    \right|  &\leq  \left[ |h|  + \frac{\left\| V \right\| }{2}  \right] a    e^{(|\im k| - \im k)x_I} e^{a} \label{second iteration Jost function}
	\end{align}
	where $\hat{V}(k) = \int_0^{x_I} e^{2ikt} V(t) dt$ is the Fourier transform of the potential $V$ and $a =\frac{ ||V||}{\max (1,|k|)}$, with $|| V || := \int_{\mathbb{R}} |V(x)| dx$.
\end{lemma}

\begin{lemma}\label{Lemma norming constants}
	If $V \in \mathbb{V}^1_{x_I}$ and if $k_1,... \,,k_N \in i\mathbb{R}_+$ are the zeros of the Jost function $f_h(k)$ such that $|k_1|> ...  > |k_N|>0$, then the following inequalities hold
	\begin{equation}\label{inequalities for the derivative of the Jost funct}
	\ii (-1)^j \dot{f}_h(k_j) >0, \quad \text{and} \quad (-1)^j f_h(-k_j) <0, \qquad \text{for} \; j=1,... \,,N,
	\end{equation}
	where the dot denotes the derivative with respect to $k$.
\end{lemma}

\section{The spectral problem}

In this section, we introduce the Weyl function formalism and we recover a Gelfand--Levitan--type equation (see Proposition \ref{Proposition Gelfand Levitan eq}) following a similar procedure as in \cite[Chapter 1, Section 1]{Novikov} and \cite{Bealsetal1995} adapted to our Love scalar boundary value problem. Then, we establish a bijection (see Theorem \ref{Theorem Inverse spectral Love problem}) between a class $\mathbb{M}_{x_I}$ of Weyl function  (see Definition \ref{Class of Weyl function}) and the class $\mathbb{V}^1_{x_I}$ (see Definition \ref{Class of potentials V' in L1}). We do not follow the usual approach in which the Weyl--Titchmarsh function is defined to be a Herglotz--Nevanlinna function and, from which by using its integral representation and the Stieltjes inversion formula, one can obtain the spectral measure  (see \cite[Theorem 9.17]{Teschl}). Instead, we follow the approach of \cite[Chapter 2]{FreilingYurko} and define the Weyl function in a different way (see Definition \ref{Weyl function definition}), which is more suitable for non self-adjoint problems. 

\subsection{Estimates of the regular solution}

We want to obtain an estimate for the regular solution $\varphi$ in the limit $k \to \infty$. We start from the Volterra-type expression for the regular solution $\varphi$
\begin{equation}\label{first eq var}
\varphi(x,k)= \cos kx - h\frac{\sin kx}{k} + \int_{0}^{x} \frac{\sin \left[ k(x-t) \right] }{k} V(t) \varphi(t,k) dt.
\end{equation}
We can easily see that this function satisfies the differential equation and the boundary condition.
Indeed
\begin{equation}\label{second eq var}
\varphi'(x,k)= -k\sin kx - h\cos kx + \int_{0}^{x} \cos \left[ k(x-t) \right] V(t) \varphi(t,k) dt
\end{equation}
and
\begin{align*}
&\varphi''(x,k)= -k^2\cos kx + hk \sin kx + V(x) \varphi(x,k) -\int_{0}^{x} k \sin \left[ k(x-t) \right]  V(t) \varphi(t,k) dt;
\end{align*}
thus $-\varphi''(x,k) + V(x) \varphi(x,k) = k^2 \varphi(x,k)$. Moreover,
\begin{equation*}
\varphi'(0,k)= -h, \quad \quad \varphi(0,k)=1,
\end{equation*}
so also the boundary condition $\varphi'(0,k) + h \varphi(0,k) = 0$ is satisfied.
Taking the absolute value of \eqref{first eq var} and since $|\sin kx|\leq \exp (|\eta| x)$ and $|\cos kx| \leq \exp (|\eta| x)$, where $\eta=\Im k$, we get
\begin{equation*}
|\varphi(x,k)| \leq  \exp (|\eta| x) + \frac{\exp (|\eta| x) }{|k|}  + \int_{0}^{x} \frac{ \exp (|\eta| (x-t))  }{|k|} |V(t)| |f(t,k)| dt.
\end{equation*}
We define $\beta_T(k) = \max_{0\leq x \leq T} (|\varphi(x,k)|) \exp (-|\eta| x) $ and we have then for $|k| > 1$
\begin{equation*}
\beta_T(k) \leq C_1 + \frac{1}{|k|} \beta_T(k) \int_{0}^{T} |V(t)| dt \leq  C_1 + \frac{1}{|k|} \beta_T(k) \int_{0}^{\infty} |V(t)| dt
\end{equation*}
which for $|k| \to \infty$ implies $\beta_T(k) = O(1)$, hence $\varphi(x,k) = O(\exp (|\eta| x))$. Substituting this estimate on \eqref{first eq var}  we get $|\varphi(x,k)| \leq C \exp (|\eta| x)$. Doing the same for the derivative of $\varphi(x,k)$, as in \eqref{second eq var}, we get 
\begin{equation}\label{asymptotic estimates on general solution}
|\varphi^{(\nu)}(x,k)| \leq C|k|^{\nu} \exp (|\eta| x), \qquad \nu=0,1, \qquad |k|\gg 1
\end{equation}
uniformly in $x$.

\subsection{Properties of Weyl function}

In this subsection, we will define the Weyl solution and the Weyl function and present their properties. These quantities enable another approach to solve the inverse problem (see \cite[Section 2.2]{FreilingYurko}) and will also enable us to recover the Gelfand--Levitan equation in an alternative way (see Subsection \ref{Gelfand Levitan section}), as the Gelfand--Levitan equation is usually recovered from the spectral measure.

In this subsection $\lambda$ and $k$ are always related via $\lambda=k^2$ defined initially for $\im k >0$. Below we give a definition of Weyl solution that uses those of the Jost solution (Definition \ref{Jost solution definition}) and the Jost function (Definition \ref{Jost function definition}) given in the previous section.
\begin{definition}[Weyl solution]\label{Weyl solution definition}
	We define the \textbf{Weyl solution} $\phi (x,\lambda)$ as the function
	\begin{equation}\label{Weyl solution}
	\phi (x,\lambda) = \frac{f(x,k)}{f_h(k)}, \qq \im k>0.
	\end{equation}
\end{definition}
This function satisfies the differential equation $-\phi'' + V\phi = \lambda \phi$ because the Jost solution does, but does not satisfy the Robin Boundary condition. In particular:
\begin{align}\label{Weyl solution eqs}
&\phi'(0,\lambda) + h \phi(0,\lambda) = 1 \\
& \phi(x,\lambda) = O (e^{i k x}) \qquad x\to \infty, \quad k \in \Sigma ,
\end{align}
where we define the set $\Sigma := \left\lbrace k \in \mathbb{C} : \im k \geq 0, k \neq 0 \right\rbrace$.
From \eqref{estimates on Jost solution} and \eqref{second iteration Jost function} in Lemma \ref{uniform bounds on Jost function} we get the  asymptotics on the Weyl solution for large $k$
\begin{equation}\label{asymptotics on phi wrt to k}
\phi^{(\nu)}(x,\lambda) = (i k)^{\nu -1} \exp(ikx) \left(1 + o\left(\frac{1}{k}  \right)  \right), \qquad \nu=0,1, \quad |k|\to \infty.
\end{equation}
The Weyl solution is uniquely determined by the differential equation $-\phi'' + V\phi = \lambda \phi$ (see \eqref{Schrod equation}) and the boundary condition \eqref{Weyl solution eqs}. 

\begin{definition}[Weyl function]\label{Weyl function definition}
	We define the \textbf{Weyl function} $M(\lambda)$ (or Weyl--Titchmarsh function) as the function
	\begin{equation*}
	M(\lambda) :=  \phi(0,\lambda) = \frac{f(0,k)}{f_h(k)}, \qq \lambda=k^2, \im k>0.
	\end{equation*}
\end{definition}

\begin{remark}
	In some other textbooks, the Weyl--Titchmarsh function for Robin boundary condition is defined as $M(\lambda)=\frac{hf'(0,k) -  f(0,k)}{f_h(k)}$ (see \cite[formula 9.52, Chapter 9]{Teschl}). This choice  entails that $M(\lambda)$ is a Herglotz--Nevanlinna function and using its properties, it is possible to obtain the spectral measure (see \cite[Theorem 9.17]{Teschl}). In our treatment, we follow the approach and the definition of \cite[Definition 2.1.69]{FreilingYurko}.
\end{remark}

\begin{remark}\label{Remark zeros of Weyl}
	The zeros of the Jost function (Definition \ref{Jost function definition}) correspond to the poles of the Weyl function (see Theorem \ref{Theorem poles of Weyl function} below). Indeed, at zeros $k=k_j$ of the Jost function, $f(0,k_j)=- \frac{1}{h} f'(0,k_j)\ne0$.
\end{remark}

\begin{remark}
	The Weyl function $M(\lambda)$ maps $f_h(k)$ to $f(0,k)$, so $M$ is the Robin-to-Dirichlet map, since the Jost function in the Dirichlet boundary value problem ($h=\infty$) is precisely $f(0,k)$ (see \cite{Korotyaev}). In the case of Dirichlet boundary condition, the Weyl function is usually defined as $\frac{f'(0,k)}{f(0,k)}$ (see \cite[formula 9.52, Chapter 9]{Teschl}) which is Herglotz--Nevanlinna and can be reconstructed by the Dirichlet and Neumann eigenvalues and resonances.
\end{remark}
From the asymptotics of the Jost solution and Jost function we obtain the asymptotics of the Weyl function, as described in the following lemma.
\begin{lemma}\label{Lemma asympt exp of Weyl function}
	Let $V \in \mathbb{V}^1_{x_I}$, then the Weyl function (see Definition \ref{Weyl function definition}) has the asymptotic expansion
	\begin{equation}\label{asymptotic expansion of the Weyl function}
	M(\lambda) = \frac{1}{ik} \left[ 1 - \frac{h}{ik} + \frac{\hat{V}(k)}{ik} +  o(k^{-1}) \right], \qq \qquad  |k| \to +\infty.
	\end{equation}
\end{lemma}
\begin{proof}
	From \eqref{second iteration Jost function}  and \eqref{estimates on Jost solution} in Lemma \ref{uniform bounds on Jost function}, we can get the asymptotics of the Weyl function
	\begin{equation*}
	M(\lambda) = \frac{1}{ik} \left( 1 + O \left( \frac{1}{k} \right) \right), \qquad |k| \to +\infty.
	\end{equation*}
	Using \eqref{second iteration Jost function}  and \eqref{estimates on Jost solution 2} we can get higher order terms of the expansion of the Weyl function in terms of $k$
	\begin{align*}
	M(\lambda) &= \left(  1 - \frac{\hat{V}(0) - \hat{V}(k) }{2ik} + o(k^{-1})   \right) \left( \frac{1}{ik + h - \frac{\hat{V}(0) + \hat{V}(k) }{2} + o(1)    } \right) \nonumber \\
	&  = \frac{1}{ik} \left(  1 - \frac{\hat{V}(0) -  \hat{V}(k) }{2ik} + o(k^{-1})   \right) \left(  1 - \frac{h}{ik} +  \frac{\hat{V}(0) + \hat{V}(k) }{2ik}  + o(k^{-1}) \right) \nonumber \\
	&= \frac{1}{ik} \left[ 1 - \frac{h}{ik} + \frac{\hat{V}(k)}{ik} +  o(k^{-1}) \right].
	\qedhere
	\end{align*}
\end{proof}
Note that we can write the Weyl solution as
\begin{equation}\label{Weyl solution formula}
\phi(x,\lambda) = \theta(x,k) + M(\lambda) \varphi(x,k)
\end{equation}
where $\varphi (x,k)$ and $\theta (x,k)$ are solution of \eqref{Schrod equation} satisfying 
\begin{align}
&\theta(0,k) = 0 \qquad \theta'(0,k) =  1 \nonumber\\
&\varphi(0,k) = 1 \qquad \varphi'(0,k) =  - h \label{condition gen sol}
\end{align}
and $\varphi (x, k)$ is the regular solution as in Definition \ref{regular solution varphi}.
We can see that

\begin{equation}\label{Wronskian between gen sol and weyl sol}
W(\varphi(x,k), \phi(x,\lambda) ) = W(\varphi(x,k), \theta(x,k) ) = 1.
\end{equation}
We denote 
\begin{align*}
&\Lambda = \left\lbrace \lambda=k^2 : k \in \Sigma, \; f_h(k)=0 \right\rbrace 
\end{align*}
and
\begin{align*}
& \Lambda' =  \left\lbrace \lambda=k^2  : \Im k> 0 , \; f_h(k)=0 \right\rbrace.
\end{align*}
The set $\Lambda'$ consists of all the eigenvalues of the differential equation $-f'' + Vf=\lambda f$ (see \eqref{Schrod equation}).
By Lemma \ref{Lemma asympt exp of Weyl function}, the Weyl function at the second order can also be written as
\begin{equation}\label{Weyl function second order}
M(\lambda) = \frac{1}{ik} \left( 1  -\frac{h}{ik} + \frac{1}{ik} \int_{0}^{\infty} V(t) e^{2ikt} dt + o \left( \frac{1}{k} \right) \right),  \quad |k| \to +\infty,k \in \Sigma.
\end{equation}
The following definition of the domain of $\lambda$ comes from \cite[Chapter 2]{FreilingYurko}.
\begin{definition}\label{Definition lambda plane}
	We define $\Pi $ as the $\lambda$-plane with the cut $\lambda \geq 0$, and $\Pi_1 = \overline{\Pi} \backslash \left\lbrace 0\right\rbrace$. $\Pi$ and $\Pi_1$ must be considered as a subset of the Riemann surface of the square-root function.
\end{definition}
In Definition \ref{Definition lambda plane} we stated that $\Pi$ and $\Pi_1$ must be considered as a subset of the Riemann surface of the square root because $\lambda$ as a square of $k$ ($k=\sqrt{\lambda}$) lays in two copies of the complex plane with cuts on the positive real axis and glued together. Hence, $\Pi$ and $\Pi_1$ live in the first (physical sheet) of these two sheets.
Since the cut is placed in the real positive axis of the $\Pi$ $\lambda$-plane, the Weyl function has a jump between above and below the cut. This motivates the following definition.

\begin{definition}[Jumps of Weyl function]\label{Jumps of Weyl function}
	We define
	\begin{equation}\label{jumps of weyl}
	T(\lambda) = \frac{1}{2 \pi i} \left( M^- (\lambda) - M^+ (\lambda) \right), \qquad \lambda > 0,
	\end{equation}
	to be the jumps of the Weyl function $M(\lambda)$ (see Definition \ref{Weyl function definition}), where
	\begin{equation*}
	M^{\pm} (\lambda)= \lim_{z\to0, \re z>0} M(\lambda \pm iz) .
	\end{equation*}
\end{definition}
From Definition \ref{jumps of weyl}, we can see that $T(\lambda)$ represents the jumps (discontinuity points of the first kind) of the Weyl function.
Thanks to \eqref{Weyl function second order} and \eqref{jumps of weyl} we get the following expansion for $T(\lambda)$:
\begin{align*}
T(\lambda) &= \frac{1}{2i\pi k} \left[ - \frac{1}{ik} \left( 2 + \frac{1}{ik} \int_{0}^{\infty} V(t) \left(e^{2ikt} - e^{-2ikt}  \right)dt + o\left(\frac{1}{k}\right) \right)\right]  \\
&=\frac{1}{\pi k } \left( 1 + \frac{1}{k} \int_{0}^{\infty} V(t) \sin 2kt \, dt + o\left(\frac{1}{k}\right) \right), \qquad   k \to +\infty.
\end{align*}
For $a>0$, we consider the points $\lambda = a \pm i 0$ in $\Pi_1$. For $\lambda=k^2$, the point $\lambda = a+i 0\in \Pi_1$ corresponds to $k=\sqrt{a+i 0} > 0$ in the positive real axis of the $k$ complex plane, while $\lambda = a-i 0\in \Pi_1$ corresponds to the point $k=\sqrt{a- i 0}<0$ situated on the negative real axis for $k$.

\begin{definition}[Spectral normalizing constant]\label{Spectral norming constant}
	We define the spectral normalizing constant $\alpha_j$  to be the complex numbers
	\begin{align*}
	\alpha_j := \Res_{\lambda= \lambda_j} M(\lambda), \qq j=1, ... \,, N
	\end{align*}
	where $\left\lbrace \lambda_j \right\rbrace_{j=1}^N = \Lambda'$.
\end{definition}
In the following proposition we connect the jump $T(\lambda)$ of the Weyl function to the Jost function.
\begin{proposition}\label{Relation between jump of the Weyl function with Jost function}
	Let $T(\lambda)$ be the jumps of the Weyl function as in Definition \ref{Jumps of Weyl function}, then
	\begin{equation}\label{Formula for the jump of Weyl function}
	T(\lambda) = \frac{k}{\pi |f_h(k)|^2}, \qq k>0.
	\end{equation}
\end{proposition}
\begin{proof}
	We follow the argument in \cite[page 134]{FreilingYurko}).
	Identity \eqref{Formula for the jump of Weyl function} holds if the following identities are true
	\begin{equation}\label{appendix 1}
	W(f(x,k),\overline{f(x,k)}) = -2ik 
	\end{equation}
	\begin{equation}\label{appendix 2}
	\overline{f(x,k)} = f(x,-k), \qquad \overline{f_h(k)} = f_h(-k).
	\end{equation}
	Those identities hold as the problem \eqref{Schrod equation}--\eqref{robin bc} with domain \eqref{domain of self adjoint operator} is self-adjoint.
	Indeed,  $k^2 -iz$ is a complex number with real part $\re (k^2) + \im z$ and imaginary part equal to $\im (k^2) -\re z$. This complex number in the $k$ complex plane corresponds to the roots $|k_z| e^{i\theta_z}$ and $|k_z| e^{i (\theta_z + \pi)}$, where 
	\begin{align*}
	&|k_z| = \left(  (\re (k^2) + \im z)^2 + (\im (k^2) -\re z)^2    \right)^{1/4} \\
	&\theta_z = \arctan \left( \frac{\im (k^2) -\re z}{2(\re (k^2) + \im z)} \right).
	\end{align*}
	In the limit $z \to 0$ along $z>0$, these two solutions converge to $k$ and $-k$ respectively.
	Hence, we have $M^-(\lambda)= \frac{f(0,-k)}{f_h(-k)}$ and 
	\begin{align*}
	T(\lambda) &= \frac{1}{2 \pi i} \left( \frac{f(0,-k)}{f_h(-k)} - \frac{f(0,k)}{f_h(k)} \right)  =  \frac{1}{2 \pi i} \left( \frac{\overline{f(0,k)}}{ \overline{f_h(k)}} - \frac{f(0,k)}{f_h(k)} \right)  \\
	&=   \frac{1}{2 \pi i} \left( \frac{ \overline{f(0,k)} ( f'(0,k) + h f(0,k))  - f(0,k) ( \overline{f'(0,k)} + h \overline{f(0,k)} )    }{|f_h(k)|^2}\right)  \\
	&=\frac{1}{2 \pi i} \left( \frac{W(\overline{f}, f) }{|f_h(k)|^2} \right)= \frac{k}{\pi |f_h(k)|^2},
	\end{align*}
	where in the second step we used \eqref{appendix 2} and in the last we used \eqref{appendix 1}.
\end{proof}
From the previous proposition, we can see that we can recover the jump function from the Jost function, but not the converse. The eigenvalues $\left\lbrace k_n\right\rbrace_{n=1}^{N}$, the spectral norming constants $\left\lbrace \alpha_n\right\rbrace_{n=1}^{N}$ and the jump function $T(\lambda)$ are usually considered in the literature as the data for the inverse spectral problem (see \cite[Definition 2.3.1]{FreilingYurko}).

The following results are useful for the inverse result at the end of this paper.

\begin{lemma}\label{proof 4}
	The following holds
	\begin{equation}\label{Asymptotics for xi to 0}
	\frac{k}{f_h(k)} = O(1), \qquad k \to 0, \; \Im k \geq 0.
	\end{equation}
\end{lemma}
\begin{proof}
	We follow the proof of \cite[Theorem 2.3.5]{FreilingYurko}.
	Since $W(f(x,k), f(x,-k))= - 2i k$ and $f_h(k) = f'(x,k) +h f(x,k)$, we have that
	\begin{align*}
	& - 2i k = f(0,k) f'(0,-k) - f'(0,k) f(0,-k)=
	f(0,k) ( f_h(-k) - h f(0,k) ) \\
	& - ( f_h(k) + h f(0,k) ) f(0,-k) = f(0,k)  f_h(-k) - f_h(k) f(0,-k).
	\end{align*}
	We set
	
	\begin{equation*}
	g(k)= \frac{2i k}{f_h(k)}
	\end{equation*}
	so, for real $k \neq 0$, we have
	\begin{equation*}
	g(k)=   f(0,-k) + S(k) f(0,k)  
	\end{equation*}
	where $S(k)= - f_h(-k)/f_h(k)$ is the scattering function. Because of the property $\overline{f_h(k)} = f_h(-k)$, we know that $f_h(k)$ and $f_h(-k)$ have the same modulus, so $|S(k)|= 1$ for real $k \neq 0$ . Let $\lambda_j = k^2_j$, $k_j = i \tau_j$, $0<\tau_1< ... <\tau_m$ and denote $\Sigma_{\tau^*}$ as
	\begin{equation*}
	\Sigma_{\tau^*} = \left\lbrace k : \Im k > 0, |k| < \tau^* \right\rbrace
	\end{equation*}
	where $\tau^* = \tau_1/2$, considering the values $k_j$ corresponding to the eigenvalues $\lambda_j$ ordered from the smallest to the largest.
	The function $g(k)$ is analytic in $\Sigma_{\tau^*}$ and continuous in $\bar{\Sigma_{\tau^*}} \, \backslash \left\lbrace 0 \right\rbrace$ and from the estimates on the Jost solution we can say that
	\begin{equation*}
	|g(k)| \leq C \qquad \text{for real } \quad k \neq 0.
	\end{equation*}
	With this last estimate, we see that $g(k)$ has a removable singularity in the origin, and consequently $g(k)$ is continuous in $\Sigma_{\tau^*}$ and \eqref{Asymptotics for xi to 0} is satisfied.
\end{proof}

\begin{proposition}\label{Connection between norming constants}
	The spectral normalizing constants $\alpha_j$ from Definition \ref{Spectral norming constant} are strictly positive and are given by
	\begin{align}\label{formula spectral norming constant}
	\alpha_j =  4k_j^2 \left[ \frac{-i}{f_h(-k_j)\dot{f}_h(k_j)}\right] > 0.
	\end{align}
\end{proposition}
\begin{proof}
	We recall the regular solution
	\begin{equation*}
	\varphi (x,k) = - \frac{1}{2\ii k} \left[f_h(-k)f(x,k) - f_h(k)f(x,-k) \right]
	\end{equation*}
	that, when $k_j$ is a zero of the Jost function, becomes the eigenfunction
	\begin{align*}
	\varphi (x,k_j) = - \frac{1}{2\ii k_j} \left[f_h(-k_j)f(x,k_j)  \right].
	\end{align*}
	We know that $\varphi(x,k)$ satisfies $\varphi(0,k)=1$ (see \eqref{condition gen sol}), hence
	\begin{equation}\label{condition norm const}
	-2\ii k_j = f_h(-k_j)f(0,k_j).
	\end{equation}
	From the definition of $\alpha_j$ we can write
	\begin{align}\label{Residue alpha j formula}
	\alpha_j = \Res_{\lambda= \lambda_j} M(\lambda) = \lim_{\lambda \to \lambda_j} \frac{(\lambda - \lambda_j)f(0,k)}{(k-k_j)\frac{d}{dk} f_h(k)} = \frac{2k_j f(0,k_j)}{\frac{d}{dk} f_h(k)|_{k=k_j}}.
	\end{align} 
	Plugging in \eqref{condition norm const}, we get
	\begin{align*}
	\alpha_j = \frac{2k_j (-2\ii k_j)}{f_h(-k_j)\frac{d}{dk} f_h(k)|_{k=k_j}}= - \frac{4 \ii k_j^2}{f_h(-k_j)\dot{f}_h(k_j)} = 4k_j^2 \left[ \frac{-i}{f_h(-k_j)\dot{f}_h(k_j)}\right] > 0
	\end{align*}
	where the last inequality follows from  \eqref{inequalities for the derivative of the Jost funct}  in Lemma \ref{Lemma norming constants} and $k_j^2$ being negative.
\end{proof}

The following theorem shows a representation formula for the Weyl function $M(\lambda)$, which can be reconstructed from the jumps $T(\lambda)$, the spectral normalizing constants $\alpha_j$ and the eigenvalues $\lambda_j$, as in \cite[Lemma 2.3.1]{FreilingYurko}.
\begin{theorem}\label{Weyl function and spectral data}
	The Weyl function is uniquely determined by the specification of the spectral data $(T(\lambda)$, $\left\lbrace \lambda_k, \alpha_k \right\rbrace_{k=1}^N)$ via the formula
	\begin{equation}\label{representation of weyl function}
	M(\lambda) = \int^{\infty}_{0} \frac{T(\mu)}{\lambda - \mu} d\mu + \sum_{k=1}^{N} \frac{\alpha_k}{\lambda - \lambda_k}, \quad \lambda \in \Pi \backslash \Lambda'.
	\end{equation}
	
\end{theorem}

\begin{proof}
	We follow the proof in \cite[Lemma 2.3.1]{FreilingYurko}.
	We consider the function 
	\begin{equation*}
	I_R (\lambda) := \frac{1}{2 \pi i} \int_{|\mu|=R} \frac{M(\mu)}{ \lambda - \mu} d\mu.
	\end{equation*}
	Since $M(\lambda)= O(k^{-1})$ for $k \to \infty$, then $ \lim_{R\to \infty} I_R (\lambda) =0$. Now, we deform the contour to avoid the singularity at $\mu=\lambda$ with the little circle $\gamma_r(\lambda)$ and to avoid the cut $\left]0,+\infty \right]$. Hence,
	
	\begin{align*}
	&\lim_{R \to 0}I_R (\lambda) = \lim_{r\to 0} \frac{1}{2 \pi i} \int_{\gamma_r(\lambda)} \frac{M(\mu)}{\lambda - \mu} d\mu + \lim_{\epsilon \to 0}\frac{1}{2 \pi i} \int^{0 - i\epsilon}_{+ \infty - i \epsilon} \frac{M(\mu)}{\lambda - \mu} d\mu \\ 
	&+ \lim_{\epsilon \to 0}\frac{1}{2 \pi i} \int_{0 + i\epsilon}^{+\infty + i \epsilon} \frac{M(\mu)}{\lambda - \mu} d\mu - 
	\frac{1}{2 \pi i} (2 \pi i) \sum_{k=1}^{m} \Res{\left(\frac{M(\mu)}{\lambda - \mu}\right)},
	\end{align*}
	where the last term is the sum of the residues of $\frac{M(\mu)}{\lambda - \mu}$ viewed as a function of $\mu$.
	In the first term we apply the residue theorem noticing that the little circle is run through in anti-clockwise direction; in the second term we make the substitution $\eta = \mu + i \epsilon$; in the third term we make the substitution $\eta = \mu - i \epsilon$, and in the last term we replace the residue of the Weyl function with $\alpha_k$ (see Definition \ref{Spectral norming constant}):
	\begin{align*}
	&0 = \frac{1}{2 \pi i} (-2 \pi i) \lim_{\mu \to \lambda} (\mu - \lambda) \frac{M(\mu)}{\lambda - \mu} + \lim_{\epsilon \to 0}\frac{1}{2 \pi i} \int_{+\infty }^{0} \frac{M(\eta - i \epsilon)}{\lambda - \eta + i \epsilon} d\eta \\
	&+ \lim_{\epsilon \to 0}\frac{1}{2 \pi i} \int_{0 }^{+\infty } \frac{M(\eta + i \epsilon)}{\lambda - \eta - i \epsilon} d\eta - \sum_{k=1}^{m} \frac{\alpha_k}{\lambda - \lambda_k}.
	\end{align*}
	Since $T(\eta) = \lim_{z\to0, \re z>0} \frac{1}{2 \pi i} \left( M(\eta - i z) - M(\eta + i z) \right) $, we can write:
	\begin{equation*}
	0= M(\lambda) + \int_{+\infty }^{0}  \frac{T(\eta)}{\lambda - \eta } d\eta - \sum_{k=1}^{m} \frac{\alpha_k}{\lambda - \lambda_k},
	\end{equation*}
	which is \eqref{representation of weyl function}.	
\end{proof}
We can write the Weyl function $M(\lambda)$ in terms of the jump function $T(\lambda)$ and the normalizing constants $\alpha_k$ through the formula
\begin{equation*}
M(\lambda) = \int^{\infty}_{0} \frac{T(\mu)}{\lambda - \mu} d\mu + \sum_{k=1}^{N} \frac{\alpha_k}{\lambda - \lambda_k}, \quad \lambda \in \Pi \backslash \Lambda'
\end{equation*}
as we can see in Theorem \ref{Weyl function and spectral data}. 
In order to reconstruct the Weyl function, we need to know the jump function, the eigenvalues and the normalizing constants.

\begin{remark}
	As we can see from Proposition \ref{Relation between jump of the Weyl function with Jost function} and Proposition \ref{Connection between norming constants}, we can retrieve uniquely the Weyl function from the Jost function $f_h$ and the eigenvalues $\left\lbrace k_j \right\rbrace_{1, \,... \,,N}$.
\end{remark}

We state below a theorem from \cite[Theorem 2.1.5]{FreilingYurko}, which proves the analyticity of $M(\lambda)$ in a certain region of the complex plane, which is related through Definition \ref{Weyl function definition} to the analyticity of the Jost function and the Jost solution (see Theorem \ref{Analyticity and continuity of Jost function and solution})
\begin{theorem}\label{Theorem poles of Weyl function}
	The Weyl function $M(\lambda)$ is analytic in $\Pi \backslash \Lambda'$ and continuous in  $\Pi_1 \backslash \Lambda$. The set of singularities of $M(\lambda)$ (as an analytic function) coincides with the set $\Lambda_0 = \left\lbrace \lambda : \lambda \geq 0 \right\rbrace \cup \Lambda$.
\end{theorem}
Next, we prove a uniqueness result for the Weyl function (see \cite[Theorem 2.2.1]{FreilingYurko}).
\begin{theorem}[Uniqueness]\label{Uniqueness}
	Let $V$ and $\tilde{V}$ be in $\mathbb{V}^1_{x_I}$ with Weyl functions $M$ and  $\tilde{M}$ respectively. If $M=\tilde{M}$, then $V = \tilde{V}$.
\end{theorem}
\begin{proof}
	We closely follow  the proof given in \cite{FreilingYurko}.
	We define the matrix $P(x,\lambda)=\left[ P_{j,k=1,2}\right]$ by the formula
	\begin{equation}\label{Matrix equation 1}
	P(x,\lambda)
	\begin{bmatrix}
	\tilde{\varphi}(x,\lambda)      & \tilde{\phi}(x,\lambda)\\
	\tilde{\varphi}'(x,\lambda)     & \tilde{\phi}'(x,\lambda)
	\end{bmatrix}
	= 
	\begin{bmatrix}
	\varphi(x,\lambda)    &\phi(x,\lambda) \\
	
	\varphi'(x,\lambda)     & \phi'(x,\lambda)
	\end{bmatrix}.
	\end{equation}
	
	Multiplying both sides of the equation by the inverse of the matrix of the left-hand side we get
	\[
	P(x,\lambda) =
	\begin{bmatrix}
	\varphi(x,\lambda)    &\phi(x,\lambda) \\
	\varphi'(x,\lambda)     & \phi'(x,\lambda)
	\end{bmatrix}
	\frac{1}{W(\tilde{\varphi}, \tilde{\phi})}
	\begin{bmatrix}
	\tilde{\phi}'(x,\lambda)     & -\tilde{\phi}(x,\lambda)\\
	-\tilde{\varphi}'(x,\lambda)     & \tilde{\varphi}(x,\lambda) 
	\end{bmatrix}.
	\]
	Since the Wronskian $W(\tilde{\varphi}, \tilde{\phi})=1$ because of \eqref{Wronskian between gen sol and weyl sol}, we can multiply the two matrices and recover the components of the matrix $P(x,\lambda)$ 
	\begin{align}\label{components of matrix P}
	&P_{j1}(x,\lambda) = \varphi^{(j-1)}(x,\lambda) \tilde{\phi}'(x,\lambda) - \phi^{(j-1)}(x,\lambda) \tilde{\varphi}'(x,\lambda) \nonumber \\
	&P_{j2}(x,\lambda) = \phi^{(j-1)}(x,\lambda) \tilde{\varphi}(x,\lambda) - \varphi^{(j-1)}(x,\lambda) \tilde{\phi}(x,\lambda).
	\end{align}
	Solving \eqref{Matrix equation 1} with respect to $\phi$ and $\varphi$ we get
	\begin{align}\label{eq num 196}
	&\varphi(x,\lambda) = P_{11}(x,\lambda) \tilde{\varphi}(x,\lambda) + P_{12}(x,\lambda) \tilde{\varphi}'(x,\lambda) \nonumber \\
	&\phi(x,\lambda) = P_{11}(x,\lambda) \tilde{\phi}(x,\lambda) + P_{12}(x,\lambda) \tilde{\phi}'(x,\lambda). 
	\end{align}
	From \eqref{asymptotics on phi wrt to k} and \eqref{asymptotic estimates on general solution}, for $|\lambda| \to \infty$ we get
	\begin{equation}\label{eq num 199}
	|P_{11}(x,\lambda) -1| \leq \frac{C}{|k|}, \qquad |P_{12}(x,\lambda)| \leq \frac{C}{|k|}, \quad |k|\to \infty.
	\end{equation}
	From \eqref{components of matrix P}, plugging the definition of the Weyl solution $\phi(x,\lambda)$ and $\tilde{\phi}(x,\lambda)$, as in  \eqref{Weyl solution formula}, we get
	\begin{align*}
	&P_{11} = \varphi(x,\lambda) \tilde{\theta}'(x,\lambda) - \theta(x,\lambda) \tilde{\varphi}'(x,\lambda) + (\tilde{M}(\lambda) - M(\lambda) ) \varphi(x,\lambda) \tilde{\varphi}'(x,\lambda) \\
	& P_{12} =  \theta(x,\lambda) \tilde{\varphi}(x,\lambda)  - \varphi(x,\lambda) \tilde{\theta}(x,\lambda)   + (\tilde{M}(\lambda) - M(\lambda) ) \varphi(x,\lambda) \tilde{\varphi}(x,\lambda).
	\end{align*}
	So, if $M(\lambda) \equiv \tilde{M}(\lambda)$, then for each fixed $x$, the functions $P_{11} (x,\lambda)$ and $P_{12} (x,\lambda)$ are entire in $\lambda$. The estimates \eqref{eq num 199} yield $P_{11} (x,\lambda)\equiv 1$ and $P_{12} (x,\lambda)\equiv 0$. Substituting this into \eqref{eq num 196} we get $\varphi(x,\lambda) \equiv \tilde{\varphi}(x,\lambda)$ and $\phi(x,\lambda) \equiv \tilde{\phi}(x,\lambda)$ for all $x$ and $\lambda$, then $V= \tilde{V}$.
\end{proof}
\begin{remark}
	Theorem \ref{Uniqueness} can be found in the literature in the case of Dirichlet boundary condition under the name of \textit{Borg-Marchenko uniqueness theorem} (see \cite{Borg}). The converse of it was proved in a local version by Barry Simon (see \cite{Simon1999} but also \cite{rybkin, Gesztesy, Bennewitz}) employing the Phragmen-Lindel{\"o}f theorem and Liouville theorem, under the assumption that if two Weyl functions asymptotically agree modulo an exponentially small function, then the two potential agree in a certain interval.
\end{remark}

\subsection{The main equation of the inverse spectral problem}\label{Gelfand Levitan section}

In this section we show an alternative way (similar to \cite{Bealsetal1995} for the Rayleigh case) to recover the Gelfand--Levitan equation, that is an integral equation from which we can reconstruct the potential $V$ and the boundary coefficient $h$ of a Schr{\"o}dinger boundary value problem. The ordinary way to obtain the Gelfand--Levitan equation is from the spectral measure. Here we obtain it through a function $\psi$, that depends on the Weyl function and is discontinuous on the real line with jumps proportional to the jumps of the Weyl function. We are motivated by the fact that in the Rayleigh problem we are not able to recover the Gelfand--Levitan equation through the spectral measure as the operator is not self-adjoint, even though we will not extend the following to the Rayleigh problem.

We recall Definition \ref{Weyl solution definition} and Definition \ref{Weyl function definition}. We define $\phi_{\pm}(x,\lambda) $ as
\begin{equation*}
\phi_{\pm}(x,\lambda) = \frac{f(x,\pm k)}{f_h(\pm k)} , \qq \im k > 0.
\end{equation*}
Note that for $\lambda>0$ we have (see also Definition \ref{jumps of weyl})
\begin{align*}
M^{\pm} (\lambda)= \lim_{z\to0, \re z>0}  \frac{f(0,\sqrt{\lambda \pm iz})}{f_h(\sqrt{\lambda \pm iz})} = \frac{f(0,\pm k)}{f_h(\pm k)}, \qq \lambda=k^2,\quad  k>0.
\end{align*}
We extend the definition of $M^{\pm}(\lambda)$ to $\lambda \in \mathbb{C}$ and note that $M^{\pm}(\lambda) = M(\lambda)$ for $\lambda \notin \left[0,\infty\right)$. In particular, it is easy to check that 
\begin{align*}
M^{\pm}(\lambda) = \frac{f(0,\pm k)}{f_h(\pm k)}, \qq \im k > 0.
\end{align*}

From \eqref{first eq var} we can find the asymptotics, as $|k|\to\infty$, in the upper half plane $\Im k>0$ for the regular solution $\varphi (x,k)$
\begin{align*}
\varphi (x,k) &= \cos kx - h \frac{\sin kx}{k} + \int_{0}^{x} \sin \left[k(x-t)\right] V(t) \varphi(t,k) dt  \\
&= \left(\frac{e^{ikx} + e^{-ikx}}{2}\right) -h \left(\frac{e^{ikx} - e^{-ikx}}{2ik}\right) + \int_{0}^{x} \sin \left(k (x-t)\right) \cos kt V(t) dt \\
&\phantom{=\;}- h \int_{0}^{x} \frac{\sin \left(k (x-t)\right)}{k} V(t) \sin kt \,dt  \\
&= e^{-ikx} \left( \frac{1}{2} +  \int_{0}^{x} \frac{e^{2ik(x-t)} - 1}{4i} V(t) dt + O\left( \frac{1}{ik} \right)  \right).
\end{align*}
We can do the same for the Jost solution $f(x,k)$ and find its asymptotics as $|k|\to \infty$ in the physical sheet, starting from \eqref{Volterra-type eq for f}
\begin{align*} 
f(x,k) &= e^{ikx} - \int_{x}^{\infty} \frac{e^{ik(x-t)} - e^{-ik(x-t)}}{2ik} V(t) e^{ikt} dt + o\left(\frac{1}{k}\right)  = e^{ikx} \left(1 - \int_{x}^{\infty} \frac{V(t)}{2ik} dt + o\left(\frac{1}{k}\right)  \right)
\end{align*}  
and hence
\begin{equation*}
f(x,-k)  = e^{-ikx} \left(1 + \int_{x}^{\infty} \frac{V(t)}{2ik}   dt + o\left(\frac{1}{k}\right) \right).
\end{equation*}
From Lemma \ref{Lemma asympt exp of Weyl function}, we have the following asymptotic expansion of $M(\lambda)$ 
\begin{equation}\label{Expansion of M}
M(\lambda)=  \frac{1}{\ii k} +\frac{1}{k^2} \left[h- \hat{V}(k)\right] + o(k^{-2}), \qq |k| \to +\infty.
\end{equation}
If $V' \in L^1\left(0,\infty\right)$, then we can integrate by parts the Fourier transform of $V$ and get
\begin{align*}
\hat{V}(k) = - \frac{V(0)}{2 \ii k} - \int_{0}^{x_I} \frac{V'(t)}{2 \ii k} e^{2 \ii kt} dt,
\end{align*}
and \eqref{Expansion of M} becomes
\begin{align*}
M(\lambda)=  \frac{1}{\ii k} +\frac{1}{k^2} \left[h- V(0)\right] + o(k^{-2}), \qq |k| \to + \infty. 
\end{align*}
The difference between $M^+(\lambda)$ and $M^-(\lambda)$ is
\begin{align}\label{difference of Weyl functions}
& \frac{2}{\ii k}  = M^+(\lambda) - M^-(\lambda) + o(k^{-2}).
\end{align}
\begin{definition}\label{Function psi}
	We define the function $\psi(x,k)$ discontinuous in the real line as
	\begin{align*}
	\psi (x,k) = 
	\begin{cases}
	-\ii k e^{ikx}\left(\phi_+(x,\lambda) +\frac{2i}{k} \varphi(x,k)  \right) \qquad &\Im k>0\\
	-\ii k e^{ikx} \phi_- (x,\lambda)     &\Im k<0
	\end{cases}
	\end{align*}
	and let $\psi_+(x,k)$ denote the restriction of $\psi (x,k)$ to the upper-half plane, and $\psi_-(x,k)$ the restriction of $\psi (x,k)$ to the lower-half plane. 
\end{definition}
One can compare Definition \ref{Function psi} with \cite[Formula 3.8]{Bealsetal1995} for the Rayleigh case.
We can see that the function $\psi$ is bounded on  $\mathbb{C}$.
We can also write the general solution $\varphi (x,k)$ in terms of $\psi_+$ and $\psi_-$ as
\begin{equation*}
2 \varphi (x,k) = e^{-ikx} \psi_+ (x,k) + e^{ikx} \psi_- (x,-k).
\end{equation*}
Since $\varphi$ is an even function of $k$, we also have
\begin{equation}\label{varphi formula}
2 \varphi (x,k) = e^{ikx} \psi_+ (x,-k) + e^{-ikx} \psi_- (x,k)
\end{equation}
and adding these last two we get
\begin{equation*}
4 \varphi (x,k) = e^{ikx} \left( \psi_+ (x,-k)  - \psi_- (x,-k)  \right)  + e^{-ikx} \left( \psi_+ (x,k) - \psi_- (x,k)  \right).
\end{equation*}
From \eqref{Weyl solution formula}, we see that 
\begin{align}\label{Difference of Weyl solutions}
 \phi_+(x,\lambda) - \phi_-(x,\lambda) &= \theta(x,k) + M^+(\lambda)  \varphi (x,k) - \theta(x,k) -  M^-(\lambda)\varphi (x,k) \nonumber  \\
& = \varphi (x,k) \left( M^+(\lambda) - M^-(\lambda)  \right).
\end{align}
Using \eqref{Difference of Weyl solutions} and \eqref{difference of Weyl functions} we can calculate
\begin{align}\label{psi+ meno psi-}
\psi_+(x,k) - \psi_-(x,k) &= e^{ikx} \left( \phi_+(x,\lambda) - \phi_-(x,\lambda) -\frac{2}{\ii k} \varphi (x,k) \right) (-\ii k) \nonumber \\
&= e^{ikx} \varphi (x,k) \left( M^+(\lambda) - M^-(\lambda) -\frac{2}{i k} \right) (-\ii k)  \nonumber \\
&= -ik e^{ikx} \varphi (x,k)  \left( j(k) - j(-k)  \right),
\end{align}
where $j(k)$ is of order $O(k^{-2})$, as is the coefficient of the second leading order of $M(\lambda)$, and it is defined as
\begin{equation*}
j(\pm k) :=    M^{\pm}(\lambda) \mp \frac{1}{\ii k}, \qq \lambda=k^2, \quad \im k>0,
\end{equation*}
and
\begin{equation*}
j(\pm k) :=    M^{\pm}(\lambda) \mp \frac{1}{\ii k}, \qq \lambda=k^2, \quad k>0.
\end{equation*}
The difference $j(k) - j(-k) $ is 
\begin{equation*}
j(k) - j(-k) =  
\begin{cases}
-\frac{2}{\ii k} \qq \qq &\im k>0, \\
M^+(\lambda) - M^-(\lambda) -\frac{2}{\ii k} & k>0,
\end{cases}
\end{equation*}
and it is of order $O(k^{-1})$.
Now, we want to calculate the asymptotics for  $\psi_-$:
\begin{align}\label{Asymptotics of psi-}
 \psi_-(x,k) &= -\ii k e^{ikx} \phi_-(x,\lambda)  = (-\ii k) e^{ikx} \frac{f(x,-k)}{f(0,-k)} M^-(\lambda) \nonumber  \\
& =(-ik) e^{ikx} \; \frac{e^{-ikx}\left( 1 + \int_{x}^{x_I} \frac{V(t)}{2ik} dt +o(k^{-1})\right) }{ 1 + \int_{0}^{x_I} \frac{V(t)}{2ik} dt +o(k^{-1})} \frac{1}{-ik}\left( 1 + \frac{h}{ik}  - \frac{\hat{V}(-k)}{ik} + o(k^{-1}) \right) \nonumber \\
& = 1 + \int_{x}^{x_I} \frac{V(t)}{2ik} dt  + \frac{h}{ik} - \int_{0}^{x_I} \frac{V(t)}{2ik} dt - \frac{\hat{V}(-k)}{ik} + o(k^{-1})  \nonumber  \\
& =1 - \int_{0}^{x} \frac{V(t)}{2ik} dt + \frac{h}{ik} + o(k^{-1})
\end{align}
where we used that
\begin{equation*}
\frac{f(x,k)}{f(0,k)} = e^{ikx}  \frac{\left( 1 - \int_{x}^{\infty}  \frac{V(t)}{2ik} +o(k^{-1}) \right)}{\left( 1 - \int_{0}^{\infty}  \frac{V(t)}{2ik} +o(k^{-1}) \right)}.
\end{equation*}


In the following proposition (see \cite[Proposition 4]{Bealsetal1995} for the Rayleigh case) we represent the function $\psi(x,k) $ in terms of some coefficients of the asymptotcs of the Weyl function $M(\lambda)$ and its residues. 
\begin{proposition}
	The function $\psi(x,k)$ satisfies
	\begin{align}\label{representation of psi prop 3}
	\psi(x,k) &= 1 - \frac{1}{2 \pi } \int_{-\infty}^{\infty} \frac{k' e^{ik'x} \varphi(x,k') \left(j(k') - j(-k' ) \right)}{k'-k} dk' \nonumber \\
	&\phantom{=\;} + \sum_{j=1}^N  \frac{\alpha_j}{2i(k_j + k)} e^{-ik_j x} \varphi(x,k_j)  +  \sum_{j=1}^N \frac{\alpha_j}{2i(k - k_j)} e^{ik_j x} \varphi(x,k_j).
	\end{align}
	Moreover, the limit value $\psi_{\pm} (x,k) = \psi(x, k \pm i 0)$ determines $\varphi$ in \eqref{Weyl solution formula} by
	\begin{equation}\label{varphi from boundary values}
	2 \varphi (x,k) = e^{ikx} \psi_+ (x,-k) + e^{-ikx} \psi_- (x,k)
	\end{equation}
\end{proposition}
\begin{proof}
	We consider
	\begin{equation*}
	\frac{1}{2 \pi i} \int_{-R}^{R} \frac{- \psi_+(k') + \psi_- (k') }{k'-k} dk' 
	\end{equation*}
	which can be written as
	\begin{align*}
	& - \frac{1}{2 \pi i} \int_{-R}^{R} \frac{ \psi_+(k') - 1 }{k'-k} dk' + \frac{1}{2 \pi i} \int_{-R}^{R} \frac{ \psi_- (k') -1 }{k'-k} dk'   \\
	& = \lim_{\epsilon \to 0^+}  \left(  - \frac{1}{2 \pi i} \int_{-R+i \epsilon}^{R + i\epsilon} \frac{ \psi_+(k') -1 }{k'-k} dk' - \frac{1}{2 \pi i} \int_{R - i\epsilon}^{-R - i \epsilon} \frac{ \psi_- (k') -1}{k'-k} dk'  \right).
	\end{align*}
	\begin{figure}[ht]
		\centering
		\includegraphics[width=0.5\columnwidth]{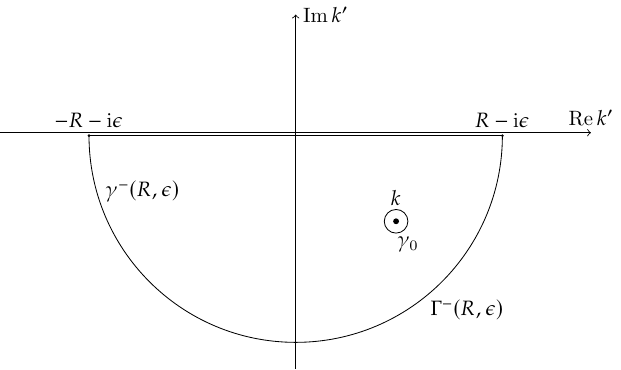}
		\caption{The closed contour $\gamma^-(R,\epsilon)$ is made by the segment from $-R-\ii \epsilon$ to $R-\ii \epsilon$ plus the arc between them in  anti-clockwise way. The arc $\Gamma^-(R,\epsilon)$  is an arc from $-R-\ii \epsilon$ to $R-\ii \epsilon$ in a clockwise way. The circle $\gamma_0$ is a contour around the pole $k'=k$.}
		\label{Contour}
	\end{figure}	
	Then we can write the integral over the interval $(-R+i\epsilon, R + i \epsilon)$ as an  integral over the contour $\gamma^+(R,\epsilon)$, which consists of the arc on the upper half plane subtended by the segment $(-R+i\epsilon, R + i \epsilon)$ plus the segment itself; an integral over the arc mentioned before with opposite verse $\Gamma^+(R,\epsilon)$ . We do something similar with the integral over the interval $(R-i\epsilon, -R - i \epsilon)$ that we write as integral over the contour $\gamma^-(R,\epsilon)$, which consists of the arc on the lower half plane subtended by the segment $(R-i\epsilon, -R - i \epsilon)$ plus the segment itself; an integral over the  arc mentioned above with opposite verse $\Gamma^-(R,\epsilon)$ and an integral over the positive oriented small circle $\gamma_0$ around the pole $k'=k$, that we consider lying in $\Im k <0$, see Figure \ref{Contour}. 
	
	Then
	\begin{align*}
	&- \frac{1}{2 \pi i} \int_{-R+i \epsilon}^{R + i\epsilon} \frac{ \psi_+(k') -1 }{k'-k} dk' - \frac{1}{2 \pi i} \int_{R - i\epsilon}^{-R - i \epsilon} \frac{ \psi_- (k') -1}{k'-k} dk' - \frac{1}{2 \pi i} \int_{\gamma^+(R,\epsilon)} \frac{ \psi_+(k') -1 }{k'-k} dk'  \nonumber \\
	& - \frac{1}{2 \pi i} \int_{\gamma^-(R,\epsilon)} \frac{ \psi_-(k') -1 }{k'-k} dk' + \frac{1}{2 \pi i} \int_{\Gamma^+(R,\epsilon)} \frac{ \psi_+(k') -1 }{k'-k} dk'  + \frac{1}{2 \pi i} \int_{\Gamma^-(R,\epsilon)} \frac{ \psi_-(k') -1 }{k'-k} dk' \nonumber \\
	&-\frac{1}{2 \pi i}  \int_{\gamma_0 } \frac{ \psi_- (k')-1 }{k'-k} dk'
	\end{align*}
	becomes
	\begin{align}\label{integral equation}
	& -\psi_-(x,k) + 1 - \sum_{j=1}^N \Res_{k'=k_j} \frac{ \psi_+(k')  }{k_j-k}   + \sum_{j=1}^N \Res_{k'=-k_j} \frac{ \psi_- (k') }{k_j+k} \nonumber \\
	&+ \frac{1}{2  \pi i} \int_{\Gamma^+(R,\epsilon)} \frac{ \psi_+(k') - 1 }{k'-k} dk' + \frac{1}{2 \pi i} \int_{\Gamma^-(R,\epsilon)} \frac{ \psi_-(k') - 1 }{k'-k} dk'.
	\end{align}
	We have
	\begin{equation*}
	\lim_{R \to \infty} \lim_{\epsilon \to 0} \left( \frac{1}{2 \pi i} \int_{\Gamma^+(R,\epsilon)} \frac{ \psi_+(k') - 1 }{k'-k} dk' + \frac{1}{2 \pi i} \int_{\Gamma^-(R,\epsilon)} \frac{ \psi_-(k') - 1 }{k'-k} dk'\right) = 0,
	\end{equation*}
by the Jordan Lemma, since $\psi_{\pm} - 1$ is of order $1/k'$.
	Thus \eqref{integral equation} becomes 
	\begin{align*}
	& - \frac{1}{2 \pi i} \int_{-\infty}^{\infty} \frac{ \psi_+(k')  }{k'-k} dk' +\frac{1}{2 \pi i} \int_{-\infty}^{\infty} \frac{ \psi_- (k') }{k'-k} dk'  = -\psi_-(x,k) + 1 \\
	& +\sum_{j=1}^N \Res_{k'=-k_j} \frac{ \psi_- (k') }{k_j+k}   - \sum_{j=1}^N \Res_{k'=k_j} \frac{ \psi_+(k')  }{k_j-k}
	\end{align*}
	so
	\begin{align}\label{integral equation 2}
	\psi(x,k)  - 1 &= - \frac{1}{2 \pi i} \int_{-\infty}^{\infty} \frac{ - \psi_+(k') + \psi_- (k') }{k'-k} dk' + \sum_{j=1}^N \Res_{k'=-k_j} \frac{ \psi_- (k') }{k_j+k}  - \sum_{j=1}^N \Res_{k'=k_j} \frac{ \psi_+(k')  }{k_j-k} \nonumber \\
	& = - \frac{1}{2 \pi i} \int_{-\infty}^{\infty} \frac{ - \psi_+(k') + \psi_- (k') }{k'-k} dk'  + \sum_{j=1}^N  \frac{ik_j}{k_j + k} e^{-ik_j x} \varphi(x,-k_j) \Res_{k'=-k_j} M^-(\lambda) \nonumber\\
	&\phantom{=\;}+  \sum_{j=1}^N \frac{ik_j}{k_j - k} e^{ik_j x} \varphi(x,k_j) \Res_{k'=k_j} M^+(\lambda).
	\end{align}
	Since $\varphi(x,k)$ is even in $k$ we have $\varphi(x,-k) = \varphi(x,k)$.
	On the left-hand side, we have a meromorphic function minus its singular terms. Recalling \eqref{Residue alpha j formula}, we have that
	\begin{equation*}
	\alpha_j = \Res_{\lambda'=\lambda_j}M(\lambda') = 2 k_j \Res_{k'=k_j} M^+(\lambda) = -2 k_j \Res_{k'=-k_j} M^-(\lambda).
	\end{equation*}
	Plugging \eqref{psi+ meno psi-} in \eqref{integral equation 2} we get
	\begin{align*}
	\psi(x,k) &= 1 - \frac{1}{2 \pi } \int_{-\infty}^{\infty} \frac{k' e^{ik'x} \varphi(x,k') \left(j(k') - j(-k' ) \right)}{k'-k} dk'  \\
	&\phantom{=\;} - \sum_{j=1}^N  \frac{i\alpha_j}{2(k_j + k)} e^{-ik_j x} \varphi(x,k_j)  -  \sum_{j=1}^N \frac{i\alpha_j}{2(k - k_j)} e^{ik_j x} \varphi(x,k_j).
	\end{align*}
	Formula \eqref{varphi from boundary values} was obtained before in \eqref{varphi formula}.
\end{proof}

In the next corollary (see \cite[Corollary, page 6701]{Bealsetal1995}) we show the connection between the potential $V$ and the eigenvalues $k_j$, the functions $j(k)$ and the normalizing constants $\alpha_j$.
\begin{corollary}
	The potential $V$ satisfies the identity
	\begin{align}\label{formula for the potential and h}
	&\int_{0}^{x} V(t) dt - 2h = -\frac{2i}{ \pi}  \int_{-\infty}^{\infty} k'  \varphi (x,k')  \cos (k'x)  \,j(k') dk' -2 \sum_{j=1}^N \alpha_j  \varphi(x,k_j) \cos (k_jx)
	\end{align}
	and the function $\varphi(x,k)$ satisfies 
	\begin{align}\label{representation of varphi function}
	 2\varphi(x,k) &= 2 \cos(kx) - \frac{i}{\pi} \int_{-\infty}^{+\infty} k' \varphi(x,k')  j(k') \left[ \frac{\sin(k'-k)x}{k'-k}  +  \frac{\sin(k'+k)x}{k'+k}\right] \nonumber\\
	&\phantom{=\;}-  \sum_{j=1}^N \alpha_j  \varphi(x,k_j) \left[ \frac{\sin(k_j-k)x}{k_j-k}  +  \frac{\sin(k_j+k)x}{k_j+k}\right].
	\end{align}
\end{corollary}

\begin{proof}
	We start first with the proof of equation \eqref{formula for the potential and h} (\textit{Step 1}) and then we recover \eqref{representation of varphi function} (\textit{Step 2}). 
	\begin{itemize}
		\item 		\textit{Step 1}.
		We already know the asymptotic expansion of $\psi$ from \eqref{Asymptotics of psi-}:
		\begin{equation*}
		\psi (x,k) - 1=  -\frac{1}{2ik} \int_{0}^{x} V(t) dt + \frac{h}{ik} + o(k^{-1}).
		\end{equation*}	
		Multiplying $\psi -1$ by $ik$ and taking the limit as $k \to \infty$, we get
		\begin{equation}\label{corollary one first eq}
		\lim_{k \to \infty} ik \left( \psi(x,k) -1 \right)= - \frac{1}{2}\int_{0}^{x} V(t) dt + h.
		\end{equation}
		In \eqref{representation of psi prop 3}, we multiply by $ik$ and take the limit as $k \to \infty$ 
		\begin{align}\label{psi - 1 representation}
		\lim_{k \to \infty} ik \left( \psi(x,k) -1 \right) &= \frac{i}{2 \pi} \int_{-\infty}^{\infty} k' e^{i k' x} \varphi (x,k') (j(k') - j(-k')) dk' \nonumber \\
		&\phantom{=\;}+ \frac{1}{2} \sum_{j=1}^N \alpha_j e^{-ik_j x} \varphi(x,k_j) + \frac{1}{2} \sum_{j=1}^N \alpha_j e^{ik_j x} \varphi(x,k_j)  \nonumber \\
		&= \frac{i}{2 \pi} \int_{-\infty}^{\infty} k' e^{i k' x} \varphi (x,k')  (j(k') - j(-k')) dk' + \sum_{j=1}^N \alpha_j  \varphi(x,k_j) \cos (k_jx).
		\end{align}
		The first term can be written as
		\begin{align*}
		&\int_{-\infty}^{\infty} k' e^{i k' x} \varphi (x,k') (j(k') - j(-k')) dk'  = \int_{-\infty}^{\infty} 2k'  \varphi (x,k') \cos (k_jx)  \,j(k') dk'.
		\end{align*}
		Plugging this result in \eqref{psi - 1 representation} and comparing with \eqref{corollary one first eq} we get
		\begin{align*}
		\int_{0}^{x} V(t) dt - 2h &= -\frac{2i}{ \pi}  \int_{-\infty}^{\infty} k'  \varphi (x,k') \cos (k'x)  \,j(k') dk' -2 \sum_{j=1}^N \alpha_j  \varphi(x,k_j) \cos (k_jx).
		\end{align*}
		\item \textit{Step 2}. We know that
		\begin{equation*}
		2 \varphi(x,k) = e^{ikx} \psi_+(x,-k) + e^{-ikx} \psi_-(x,k).
		\end{equation*}
		Both the function $\psi_+(x,-k)$ and $\psi_-(x,k)$ have poles in the lower half plane, so we can consider $k$ in the lower half plane and use the formula \eqref{representation of psi prop 3} for $e^{ikx} \psi_+(x,-k)$:
		\begin{align*}
		 e^{ikx} \psi_+(x,-k) &= e^{ikx} - \frac{1}{2 \pi } \int_{-\infty}^{\infty} \frac{k' e^{i(k + k')x} \varphi(x,k') \left(j(k') - j(-k' ) \right)}{k'+k} dk'  \\
		&\phantom{=\;}+ \sum_{j=1}^N  \frac{\alpha_je^{-i(k_j-k) x}}{2i(k_j - k)}  \varphi(x,k_j)  -  \sum_{j=1}^N \frac{\alpha_je^{i(k_j +k)x}}{2i(k + k_j)}  \varphi(x,k_j).
		\end{align*}
		The first integral can be rewritten as
		\begin{align*}
		& - \frac{1}{2 \pi } \int_{-\infty}^{+\infty} \frac{k' e^{i(k + k')x} \varphi(x,k') \left(j(k') - j(-k' ) \right)}{k'+k} dk' = \\
		&- \frac{1}{2 \pi } \int_{-\infty}^{+\infty} \frac{k' e^{i(k + k')x} \varphi(x,k') j(k')  }{k'+k} dk' -  \frac{1}{2 \pi } \int_{+\infty}^{-\infty} \frac{(-k') e^{-i(k'-k)x} \varphi(x,k') (- j(k' )) }{k-k'} (-dk') 
		\end{align*}
		where the second integral after the change of variable from $k'$ to $-k'$ becomes
		\begin{align*}
		\frac{1}{2 \pi } \int_{-\infty}^{+\infty} \frac{k' e^{-i(k'-k)x} \varphi(x,k')  j(k') }{k'-k} dk'.
		\end{align*}
		So, in the end we have
		\begin{align}\label{exp ikx times psi+}
		e^{ikx} \psi_+(x,-k) &= e^{ikx} - \frac{1}{2 \pi } \int_{-\infty}^{+\infty} \frac{k' e^{i(k + k')x} \varphi(x,k') j(k')  }{k'+k} dk' + \frac{1}{2 \pi } \int_{-\infty}^{+\infty} \frac{k' e^{-i(k'-k)x} \varphi(x,k')  j(k') }{k'-k} dk' \nonumber \\
		&\phantom{=\;}+ \sum_{j=1}^N  \frac{\alpha_je^{-i(k_j-k) x}}{2i(k_j - k)}  \varphi(x,k_j)  -  \sum_{j=1}^N \frac{\alpha_je^{i(k_j +k)x}}{2i(k + k_j)}  \varphi(x,k_j).
		\end{align}
		Similarly, for $e^{-ikx} \psi_-(x,k)$ we get
		\begin{align}\label{exp -ikx times psi-}
		e^{-ikx} \psi_-(x,k) &= e^{-ikx} - \frac{1}{2 \pi } \int_{-\infty}^{+\infty} \frac{k' e^{i(k' - k)x} \varphi(x,k') j(k')  }{k' - k} dk'  +  \frac{1}{2 \pi } \int_{-\infty}^{+\infty} \frac{k' e^{-i(k'+k)x} \varphi(x,k')  j(k') }{k'+k} dk' \nonumber \\
		&\phantom{=\;}- \sum_{j=1}^N  \frac{\alpha_j e^{i(k_j - k)x}}{2i( k_j - k)}  \varphi(x,k_j)  +  \sum_{j=1}^N \frac{\alpha_je^{-i(k_j + k)x}}{2i(k + k_j)}  \varphi(x,k_j).
		\end{align}
		Summing \eqref{exp -ikx times psi-} and \eqref{exp ikx times psi+} we get
		\begin{align*}
		 e^{ikx} \psi_+(x,-k) + e^{-ikx} \psi_-(x,k) &= 2 \cos(kx) \\
		&\phantom{=\;}- \frac{i}{\pi} \int_{-\infty}^{+\infty} k' \varphi(x,k')  j(k') \left[ \frac{\sin(k'-k)x}{k'-k}  +  \frac{\sin(k'+k)x}{k'+k}\right] \\
		&\phantom{=\;}-  \sum_{j=1}^N \alpha_j  \varphi(x,k_j) \left[ \frac{\sin(k_j-k)x}{k_j-k}  +  \frac{\sin(k_j+k)x}{k_j+k}\right].
		\qedhere
		\end{align*}
	\end{itemize}
\end{proof}
Equation \eqref{formula for the potential and h} motivates introducing
\begin{equation}\label{formula for the transformation kernel}
K(x,y) = \frac{i}{\pi} \int_{-\infty}^{+\infty} k' \varphi(x,k') j(k') \cos (k'y) dk' + \sum_{j=1}^N \varphi(x,k_j) \alpha_j \cos(k_jy).
\end{equation} 
Then equation \eqref{representation of varphi function} can be written as
\begin{equation*}
\varphi (x,k) = \cos (kx) - \frac{1}{2} \int_{-x}^{x} K(x,t) \cos (kt) dt =  \cos (kx) - \int_{0}^{x} K(x,t) \cos (kt) dt.
\end{equation*}

The next proposition shows the Gelfand--Levitan equation and the algorithm one can use to recover the potential. One can compare the following proposition with \cite[Proposition 5]{Bealsetal1995}.

\begin{proposition}\label{Proposition Gelfand Levitan eq}
	The potential can be reconstructed from the Weyl function through the formula
	\begin{equation}\label{formula potential V and K(x,x)}
	V(x) = -2 \frac{d}{dx} K(x,x),
	\end{equation}
	where $K(x,y)$ satisfies the Gelfand--Levitan equation
	\begin{equation}\label{Gelfand Levitan equation}
	K(x,y) - g(x,y)+ \frac{1}{2} \int_{-x}^{x} K(x,s) g(s,y) ds=0,
	\end{equation}
	with
	\begin{align}\label{formula for the gelfand kernel}
	g(x,y) =
	\begin{cases}
	\frac{i}{\pi} \int_{-\infty}^{+\infty} k' \cos (k'x) j(k') \cos (k'y) dk' + \sum_{j=1}^N \cos(k_jx) \alpha_j \cos(k_jy) \qquad &x\geq y \\
	0 \qqq \qqq \qqq \qqq \qqq   &x<y
	\end{cases}.
	\end{align}
\end{proposition}
\begin{proof}
	We consider 
	\begin{align}\label{kernel k times cos}
	&\int_{-x}^{x} K(x,y) \cos (ky) dy = \int_{-x}^{x} \left[\frac{i}{\pi} \int_{-\infty}^{+\infty} k' \varphi(x,k') j(k') \cos (k'y) dk' + \sum_{j=1}^N \varphi(x,k_j) \alpha_j \cos(k_jy)\right] \cos (ky) dy \nonumber \\
	&= \frac{i}{\pi} \int_{-\infty}^{+\infty} k' \varphi(x,k') j(k') + \int_{-x}^{x}\cos (k'y)\cos (ky) dydk' + \sum_{j=1}^N \varphi(x,k_j) \alpha_j  \int_{-x}^{x}\cos (k_j y)\cos (ky) dy 
	\end{align}
	and we can calculate
	\begin{align*}
	\int_{-x}^{x}\cos (\alpha y)\cos (\beta y) dy &=  \frac{\sin\left((\alpha+\beta)x\right) }{\alpha+\beta } + \frac{\sin\left((\alpha-\beta)x\right) }{\alpha-\beta}.
	\end{align*}
	Plugging this in \eqref{kernel k times cos} we get
	\begin{align*}
	\int_{-x}^{x} K(x,y) \cos (ky) dy &= \frac{i}{\pi} \int_{-\infty}^{+\infty} k' \varphi(x,k') j(k') \left[\frac{\sin\left((k'+k)x\right) }{k'+k }  + \frac{\sin\left((k'-k)x\right) }{k'-k}\right] dk'  \\
	&\phantom{=\;}+\sum_{j=1}^N \varphi(x,k_j) \alpha_j \left[\frac{\sin\left((k_j+k)x\right) }{k_j+k } + \frac{\sin\left((k_j-k)x\right) }{k_j-k}\right].
	\end{align*}
	Then, comparing with \eqref{representation of varphi function}, it follows that
	\begin{equation}\label{Remark 2 important formula}
	2\varphi(x,k) - 2 \cos(kx) = - \int_{-x}^{x} K(x,y) \cos (ky) dy.
	\end{equation}
	Taking into account \eqref{formula for the gelfand kernel} and \eqref{formula for the transformation kernel} we can calculate the difference
	\begin{align*}
	&2g(x,y) - 2K(x,y) = \frac{2i}{\pi} \int_{-\infty}^{+\infty} k' \cos (k'x) j(k') \cos (k'y) dk' + 2\sum_{j=1}^N \cos(k_jx) \alpha_j \cos(k_jy) \\
	&- \frac{2i}{\pi} \int_{-\infty}^{+\infty} k' \varphi(x,k') j(k') \cos (k'y) dk' -2 \sum_{j=1}^N \varphi(x,k_j) \alpha_j \cos(k_jy) \\
	&=  2\sum_{j=1}^N \cos(k_jy) \alpha_j \left[\cos (k_jx) - \varphi(x,k_j)\right] + \frac{2i}{\pi} \int_{-\infty}^{+\infty} k' j(k') \cos (k'y) \left[\cos (k'x) - \varphi(x,k')\right] dk' 
	\end{align*}
	and, using formula \eqref{Remark 2 important formula}, we get that
	\begin{align*}
	\cos (k'x) - \varphi(x,k') = \frac{1}{2} \int_{-x}^{x} K(x,s) \cos (ks) ds.
	\end{align*}
	Plugging this result into the previous calculations, we get
	\begin{align*}
	2g(x,y) - 2K(x,y) &= \frac{i}{\pi} \int_{-\infty}^{+\infty} \int_{-x}^{x} k' j(k') \cos (k'y) K(x,s) \cos (ks) ds \\
	&\phantom{=\;} +\sum_{j=1}^N \cos(k_jy) \alpha_j   \int_{-x}^{x} K(x,s) \cos (ks) ds \\
	&=  \int_{-x}^{x} K(x,s) \left[  \frac{i}{\pi} \int_{-\infty}^{+\infty} k' \cos (k'y) j(k') \cos (k's) dk' + \sum_{j=1}^N \alpha_j  \cos(k_jy) \cos(k_js) \right]ds \\
	&= \int_{-x}^{x} K(x,s) g(s,y) ds \qqq \text{for } \; y \leq s \leq x.
	\end{align*}
	Hence, the kernel $K(x,y)$ satisfies 
	\begin{equation*}
	K(x,y) - g(x,y)+ \frac{1}{2} \int_{-x}^{x} K(x,s) g(s,y) ds = 0.
	\qedhere
	\end{equation*}
\end{proof}
\begin{remark}[Uniqueness]
	Let $V$ and $\tilde{V}$ be in $\mathbb{V}^1_{x_I}$ with Weyl functions $M$ and  $\tilde{M}$ respectively. If $M=\tilde{M}$, then $V= \tilde{V}$. Indeed, from equation \eqref{formula for the transformation kernel} we see that
	\begin{align*}
	&K(x,x) = \frac{i}{\pi} \int_{-\infty}^{+\infty} k' \varphi(x,k') \left(M(\lambda)\ - \frac{1}{ik'} \right) \cos (k'x) dk'  + \sum_{j=1}^N \varphi(x,k_j) \cos(k_jx) 2 k_j \operatorname{Res}_{k'=k_j} M(\lambda)
	\end{align*}
	and
	\begin{align*}
	& \tilde{K}(x,x) = \frac{i}{\pi} \int_{-\infty}^{+\infty} k' \tilde{\varphi}(x,k') \left(\tilde{M}(\lambda)\ - \frac{1}{ik'} \right) \cos (k'x) dk'  + \sum_{j=1}^N \tilde{\varphi}(x,k_j) \cos(k_jx) 2 k_j \operatorname{Res}_{k'=k_j} \tilde{M}(\lambda).
	\end{align*}
	If $M (\lambda)= \tilde{M}(\lambda)$ then $ \varphi(x,k') = \tilde{\varphi}(x,k')$, which also implies $K(x,x) = \tilde{K}(x,x)$, which leads to $V(x) = \tilde{V}(x)$.
\end{remark}
\begin{remark}
	Since $K(x,s)$ and $g(s,y)$ are even, namely $K(x,-s)= K(x,s)$ and $g(-s,y) = g(s,y)$, we can write \eqref{Gelfand Levitan equation} also as 
	\begin{equation*}
	K(x,y) - g(x,y)+  \int_{0}^{x} K(x,s) g(s,y) ds=0.
	\end{equation*}
\end{remark}

The next theorem shows for which condition the Gelfand--Levitan equation \eqref{Gelfand Levitan equation} has a unique solution (see \cite[Remark (ii), page 6708]{Bealsetal1995}).
\begin{theorem}
	The Gelfand--Levitan equation \eqref{Gelfand Levitan equation} has a unique solution, for fixed $x>0$, if
	\begin{equation}\label{condition for the unique solvability of the GL equation}
	\int_{0}^{x}  \sup_{0\leq s \leq t} \left| g(t,s)\right| dt < \infty
	\end{equation}
	holds.
\end{theorem}
\begin{proof}
	The equation \eqref{Gelfand Levitan equation} is an inhomogeneous Volterra equation, where the inhomogeneous term is $-g(x,y)$. In order to have unique solvability of \eqref{Gelfand Levitan equation}, we require that the homogeneous equation
	\begin{equation*}
	K(x,y) +  \int_{0}^{x} K(x,s) g(s,y) ds=0
	\end{equation*}
	only admits the trivial solution $K(x,s)=0$. 
	One can find the solution of \eqref{Gelfand Levitan equation} from the resolvent $R(s,t)$, which is obtained by iterating the kernel $g(x,y)$
	\begin{equation*}
	R(s,t) = \sum_{k=0}^{\infty} (-1)^k g_{k+1} (s,t)
	\end{equation*}
	where $g_{k+1} (s,t)$ represents the $k+1$ iterate of the Volterra kernel. The solution is then
	\begin{equation*}
	K(x,y) = g(x,y) - \int_{0}^{x} R(x,t) g(t,y) dt.
	\end{equation*}
	We consider the second iterate of the kernel $g(x,y)$
	\begin{equation*}
	g_2(x,y) = \int_{0}^{x} \int_{0}^{t} g(t,s) g(s,y) dt ds, \qq 0\leq y \leq s \leq t \leq x.
	\end{equation*}
	Since $g(t,s)=0$ for $s>t$, we have
	\begin{align*}
	|g_2(x,y) |&=\left|  \int_{0}^{x} \int_{0}^{t} g(t,s) g(s,y) dt ds \right| \leq \int_{0}^{x} \int_{0}^{t} \sup_{0\leq s  \leq t} \left| g(t,s)\right| \sup_{0\leq y \leq s} \left| g(s,y)\right|dt ds.
	\end{align*}
	We define $d(t) := \sup_{0\leq s \leq t} \left| g(t,s)\right|$, so we get
	\begin{equation*}
	|g_2(x,y)| \leq \int_{0}^{x} \int_{0}^{t} d(s) d(t) dt ds= \frac{1}{2} \left(  \int_{0}^{x} d(s) ds \right)^2.
	\end{equation*}
	Similarly,
	\begin{equation*}
	|g_k(x,y) |\leq \frac{1}{k!} \left(  \int_{0}^{x} d(s) ds \right)^k
	\end{equation*}
	which implies that the homogeneous equation
	\begin{equation*}
	K(x,y) = - \int_{0}^{x} K(x,s) g(s,y) ds
	\end{equation*}
	admits only the trivial solution $K(x,y)=0$ as long as 
	\begin{equation*}
	\int_{0}^{x}  \sup_{0\leq s \leq t} \left| g(t,s)\right| dt < \infty.
	\qedhere
	\end{equation*}
\end{proof}
Condition \eqref{condition for the unique solvability of the GL equation} is required to have unique solvability of the Gelfand--Levitan equation. We have
\begin{align}\label{proof about the sup |g(ts)|}
&\int_{0}^{x} \left| \int_{-\infty}^{\infty}   k \cos (kt) \cos (ks) j(k) dk + \sum_{j=1}^{N}  \cos (k_jt) \cos (k_js) \; \alpha_j \right| dt \nonumber \\
&\leq  \int_{0}^{x} \int_{0}^{\infty} \left|k\right| \left| \cos kt\right|\ \left| \cos ks\right|\ \left|j(k) - j(-k) \right| dk dt+ c_1 x \nonumber \\
& \leq \int_{0}^{x} \int_{0}^{\infty}  \frac{dkdt}{\left|k\right|^2 } + c_1 x \leq c_2 x,
\end{align}
and we see that the condition is satisfied for the Weyl function of our problem, since $x>0$ is fixed in the Gelfand--Levitan equation.

\section{The inverse problem}\label{Inverse Love problem}


In this section we present an inverse result starting from a class $\mathbb{M}_{x_I}$ of Weyl functions  to the class $\mathbb{V}^1_{x_I}$. This motivates the following definition.
\begin{definition}[Class of Weyl function]\label{Class of Weyl function}
	For fixed $h\in \mathbb{R}$, we denote by $\mathbb{M}_{x_I}$ the class of functions $M(\lambda)$ satisfying the following properties:
	\begin{enumerate}[I)]
		\item \label{Condition I}
		$M(\lambda)$ is analytic in $\Pi$ with finite number $N$ of simple poles $\lambda_j<0$ and residues $\alpha_j= \Res_{\lambda= \lambda_j} M(\lambda) >0$.
		\item \label{Condition II}
		$M(\lambda)$ is continuous in $\Pi_1 \backslash \left\lbrace \lambda_1, ... \,, \lambda_N,0 \right\rbrace$ satisfying $k M(\lambda) = O(1)$ as $k\to 0$, $\Im k>0$.
		\item  \label{Condition III}
		Let $M^{\pm} (\lambda) = \lim_{\epsilon \to 0, \re \epsilon>0} M(\lambda \pm i \epsilon)$. Then
		\begin{equation*}
		T(\lambda) := \frac{1}{2 \pi i} (M^-(\lambda) - M^+(\lambda) )>0, \qqq \lambda>0.
		\end{equation*}
		\item \label{Condition IV}
		$M(\lambda) = \frac{1}{\ii k} + \frac{h}{k^2} + \frac{V(0)}{k^2}  + o(k^{-2})$, as $|k| \to +\infty$.
		\item \label{Condition VI}
		The Gelfand--Levitan equation
		\begin{equation*}
		g(x,y) + K(x,y) + \int_{0}^{x} K(x,s) g(s,y) ds = 0
		\end{equation*}
		with 
		\begin{align*}
		&g(x,y) =
		\begin{cases}
		\frac{i}{\pi} \int_{-\infty}^{+\infty} k' \cos (k'x) j(k') \cos (k'y) dk' + \sum_{j=1}^N \cos(k_jx) \alpha_j \cos(k_jy), \qquad &x\geq y, \\
		0, \qqq \qqq \qqq \qqq \qqq   &x<y,
		\end{cases}
		\end{align*}
		and $j(k) = M(\lambda) - \frac{1}{ik}$, for any fixed $x>0$, has a unique solution $K(x,y)$ with $K(x,x)$ real, absolutely continuous and $\frac{d}{dx} K(x,x) = 0$ for $x>x_I$ and non-zero in a set of non-zero Lebesgue measure $\left(x_I-\epsilon,x_I\right)$.
	\end{enumerate}
\end{definition}

In the following theorem we characterize the class $\mathbb{V}^1_{x_I} $ by the just defined class of Weyl functions $\mathbb{M}_{x_I}$.

\begin{theorem}\label{Theorem Inverse spectral Love problem}
	The map $\mathcal{J}_h: \mathbb{V}^1_{x_I} \to   \mathbb{M}_{x_I}$ defined by $\mathcal{J}_h \left(V\right):=M$ is well-defined and bijective.
\end{theorem}
\begin{proof}
	We shall prove that, for fixed $h\in \mathbb{R}$ the map $\mathcal{J}_h$ is well-defined, that is, $\mathcal{J}_h (V)=M \in \mathbb{M}_{x_I}$ for any $V \in \mathbb{V}^1_{x_I}$.
	The Jost function $f_h(k)$ has a finite number of zeros in $\mathbb{C}_+$ and that they are all simple and pure imaginary (see \cite{levitaninverse, Sottile}). In \cite{Sottile} is proved that the Jost solution $f(x,k)$ and the Jost function $f_h(k)$ are entire in $k$, hence analytic for $\im k>0$ and continuous for $\Im k \geq 0$. Then, by definition of the Weyl function $M(\lambda)$, we can conclude that the Weyl function is analytic in $\Pi$, continuous in $\Pi_1$ except at the points where the denominator vanishes (see also Theorem \ref{Theorem poles of Weyl function}), which are the simple and pure imaginary zeros of the Jost function (see Theorem \ref{Theorem poles of Weyl function} and Remark \ref{Remark zeros of Weyl}). 
	In Proposition \ref{Connection between norming constants}, we proved that $\alpha_j >0$, hence  Condition \ref{Condition I} of the definition of the class of Weyl function is satisfied.
	
	In Lemma \ref{proof 4} we showed that $\frac{k}{f_h(k)} = O(1)$. Since $f(0,k)=1 + O(1/k)$, then $kM(\lambda)= O(1)$, which implies Condition \ref{Condition II}.
	
	Condition \ref{Condition III} is proved by Proposition \ref{Relation between jump of the Weyl function with Jost function} since for $\lambda>0$ we get $T(\lambda) = \frac{\sqrt{\lambda}}{\pi \left|f_h(k)\right|^2}>0$.
	By Lemma \ref{Lemma asympt exp of Weyl function} we can see that Condition \ref{Condition IV} is satisfied. 
	
	From Proposition \ref{Proposition Gelfand Levitan eq}, we see that Condition \ref{Condition VI} is satisfied, hence $M \in \mathbb{M}_{x_I}$ and $\mathcal{J}_h$ is well-defined.
	
	The injectivity of the map $\mathcal{J}_h$ is given by Theorem \ref{Uniqueness}. 
	
	To prove surjectivity,  we fix $M(\lambda) \in \mathbb{M}_{x_I}$ and we want to prove that there exists a $V\in \mathbb{V}^1_{x_I}$ such that $\mathcal{J}_h (V) = M(\lambda)$.
	Condition \ref{Condition I}--\ref{Condition IV} allow us to define a function $g(x,y)$\footnote{Condition \ref{Condition III} is needed because $T(\lambda)$ is the spectral measure and it must be non negative.} as in \eqref{formula for the gelfand kernel}  and $K(x,y)$ which satisfies the Gelfand--Levitan equation (see Proposition \ref{Proposition Gelfand Levitan eq}). From $K(x,y)$, solution of \eqref{Gelfand Levitan equation}, we can construct (as in \eqref{Remark 2 important formula})
	\begin{equation*}
	\varphi(x,k) = \cos kx - \int_{0}^{x} K(x,y) \cos (ky) dy
	\end{equation*}
	that is a solution to the boundary value problem \eqref{Schrod equation}--\eqref{robin bc} with $V(x) = -2 \frac{d}{dx} K(x,x)$ and $h = K(0,0)$ given.
	
	From Condition \ref{Condition VI} we know that the Gelfand--Levitan equation \eqref{Gelfand Levitan equation} has a unique solution $K(x,y)$, such that $V = - 2 \frac{d}{dx} K(x,x)$ is in the class $ \mathbb{V}^1_{x_I}$.
\end{proof}

The reader can compare Definition \ref{Class of Weyl function} and Theorem \ref{Theorem Inverse spectral Love problem} with the definition of the class $\textbf{W}$ and Theorem 2.2.5 in \cite{FreilingYurko}, which are obtained for a different class of potentials and through a different Gelfand--Levitan equation.

\begin{algorithm}
	Starting from a set of eigenvalues and resonances $\left\lbrace k_j \right\rbrace_{1}^{\infty}$ we can retrieve the potential $V_{\omega}(x)$ using the following algorithm:
	\begin{itemize}
		\item Construct the Jost function from 
		\begin{align*}
		f_h(k) =  f_h(0)  e^{ik} \lim_{R \to \infty} \prod_{|k_n|\leq R} \left(1 - \frac{k}{k_n}\right),
		\end{align*}
		where $f_h(0)$ is determined so that $f_h(k)= \ii k + O(1)$ as $k \to \infty$.
		\item From $\left\lbrace k_j \right\rbrace_{1}^{\infty}$ and $f_h(k)$ we construct the jump function $T(\lambda)$ and the normalizing constant $\alpha_k$ through formulas \eqref{Formula for the jump of Weyl function} and \eqref{formula spectral norming constant}:
		\begin{align*}
		&T(\lambda) = \frac{k}{\pi |f_h(k)|^2}, \\
		&\alpha_j =  4k_j^2 \left[ \frac{-i}{f_h(-k_j)\dot{f}_h(k_j)}\right].
		\end{align*}	
		\item Use the spectral data $\left( T(\lambda), \left\lbrace \alpha_j, \lambda_j \right\rbrace_{j=1, \,... \,, N} \right)$ to construct the Weyl function via  formula \eqref{representation of weyl function}
		\begin{equation*}
		M(\lambda) = \int^{\infty}_{0} \frac{T(\mu)}{\lambda - \mu} d\mu + \sum_{k=1}^{N} \frac{\alpha_k}{\lambda - \lambda_k}, \quad \lambda \in \Pi \backslash \Lambda'.
		\end{equation*}	
		\item Then construct $g(x,y)$ in \eqref{formula for the gelfand kernel} as in
		\begin{align*}
		&g(x,y) =
		\begin{cases}
		\frac{i}{\pi} \int_{-\infty}^{+\infty} k' \cos (k'x) j(k') \cos (k'y) dk' + \sum_{j=1}^N \cos(k_jx) \alpha_j \cos(k_jy), \qquad &x\geq y \\
		0, \qqq \qqq \qqq \qqq \qqq   &x<y
		\end{cases}.
		\end{align*}
		where $j(k):=M(\lambda) - \frac{1}{ik}$.
		
		\item  Solve the Gelfand--Levitan equation \eqref{Gelfand Levitan equation} with respect to $K(x,y)$,
		\begin{equation*}
		K(x,y) - g(x,y)+ \frac{1}{2} \int_{-x}^{x} K(x,s) g(s,y) ds=0.
		\end{equation*}
		\item  Obtain the potential from \eqref{formula potential V and K(x,x)}:
		\begin{equation*}
		V_{\omega}(x) = -2 \frac{d}{dx} K(x,x).
		\end{equation*}
	\item Obtain the shear modulus from
	\begin{equation}\label{Recovery of Lame parameter}
\hat{\mu}(x) = \frac{\hat{\mu}_I  \left( \omega^2_1 - \omega^2_2 \right)}{\omega^2_1 - \omega^2_2 - \hat{\mu}_I \left( V_{\omega_1}(x) - V_{\omega_2}(x) \right)}.
	\end{equation}
	\end{itemize}
\end{algorithm}
\begin{remark}\label{Remark shear modulus}
	The formula \eqref{Recovery of Lame parameter} for the reconstruction of the Lam\'e parameter $\hat{\mu}$ is obtained in \cite[Theorem 4.9]{Sottile}.
\end{remark}

\section*{Acknowledgements}

I want to thank my former supervisor Alexei Iantchenko for having introduced me to the topic of inverse resonance problems and for interesting discussions.

\bibliographystyle{amsplain}
\bibliography{MonographyCitations}

\end{document}